\documentclass[a4paper,11pt,legno]{amsart}
\usepackage{a4wide}
\usepackage{tikz,tikz-cd,amsmath,amsfonts,amssymb,graphicx}
\usepackage[autostyle]{csquotes} 
\usepackage{enumerate}
\usepackage[mathscr]{euscript}
\usepackage{mathtools}
\usepackage{hyperref}
\usepackage{url}
\usepackage{scalerel}
\usepackage{dsfont}
\usepackage{stackengine,wasysym}

\usepackage[noabbrev,nameinlink,capitalize]{cleveref}

\theoremstyle{plain}
\newtheorem{theorem}{Theorem}[section]
\newtheorem{corollary}[theorem]{Corollary}
\newtheorem{lemma}[theorem]{Lemma}
\newtheorem{proposition}[theorem]{Proposition}

\theoremstyle{definition}
\newtheorem{definition}[theorem]{Definition}
\newtheorem{example}[theorem]{Example}
\newtheorem{remark}[theorem]{Remark}
\newtheorem*{Rmk}{Remark}
\newtheorem{construction}[theorem]{Construction}
\newtheorem{terminology}[theorem]{Terminology}
\newtheorem{notation}[theorem]{Notation}

\makeatletter
\renewcommand{\varprojlim}{%
	\mathop{\mathpalette\varlim@{\leftarrowfill@\scriptscriptstyle}}\nmlimits@
}
\renewcommand{\varinjlim}{%
	\mathop{\mathpalette\varlim@{\rightarrowfill@\scriptscriptstyle}}\nmlimits@
}
\makeatother
\newcommand{\cate}[1]{\mathscr{#1}}
\newcommand{\s}[1]{\mathbb{S}_{\scalebox{1}{$\scriptscriptstyle #1$}}}

\newcommand{\Sh}[2]{\mathrm{Shv}({#1}; {#2})}

\newcommand{\sHom}[3]{\mathrm{\underline{Hom}}_{\scalebox{1}{$\scriptscriptstyle #1$}}(#2, #3)}
\newcommand{\Hom}[3]{\mathrm{Hom}_{\scalebox{1}{$\scriptscriptstyle #1$}}(#2, #3)}
\newcommand{\homsp}[3]{\mathrm{hom}_{\scalebox{1}{$\scriptscriptstyle #1$}}(#2, #3)}
\newcommand{\Op}[1]{\cate{U}(#1)}

\newcommand{\Sec}[2]{\Gamma(#1; #2)}

\newcommand{\cSec}[2]{\Gamma_{\scalebox{1}{$\scriptscriptstyle c$}}(#1; #2)}

\newcommand{\pf}[1]{#1_{\ast}}
\newcommand{\pb}[1]{#1^{\ast}}
\newcommand{\pfp}[1]{#1_{!}}
\newcommand{\pbp}[1]{#1^{!}}
\newcommand{\pfs}[1]{#1_{\sharp}}

\DeclareMathOperator{\Fun}{Fun}
\DeclareMathOperator*{\colim}{colim}
\DeclareMathOperator{\Pro}{Pro}
\DeclareMathOperator{\opposite}{op}
\DeclareMathOperator{\lex}{lex}

\DeclareMathOperator{\singularset}{Sing}

\DeclareMathOperator{\Corr}{Corr}

\newcommand{\op}{^{\opposite}}
\newcommand{\C}{\cate{C}}
\newcommand{\D}{\cate{D}}
\newcommand{\E}{\cate{E}}

\newcommand{\Cat}{\cate{C}\mathrm{at}}
\newcommand{\rnum}{\mathbb{R}}

\newcommand{\sing}[1]{\singularset({#1})}

\newcommand{\cohcorr}[1]{\underline{\Corr}_{{#1}}}
\newcommand{\Sing}{\mathrm{Sing}}
\newcommand{\Sp}{\mathrm{Sp}}
\newcommand{\Ab}{\mathrm{Ab}}

\newcommand{\Z}{\mathbb{Z}}

\newcommand{\Ani}{\mathrm{An}}
\newcommand{\Pic}{\mathrm{Pic}}
\newcommand{\LCH}{\mathrm{LCH}}
\newcommand{\Mod}{\mathrm{Mod}}
\newcommand{\bE}{\mathbb{E}}
\newcommand{\acc}{\mathrm{acc}}
\newcommand{\CAlg}{\mathrm{CAlg}}
\newcommand{\comp}{\mathrm{comp}}
\newcommand{\sh}{\mathrm{sh}}
\newcommand{\bS}{\mathbb{S}}
\newcommand{\Tor}{\mathrm{Tor}}

\providecommand{\Set}{\mathrm{Set}}

\providecommand{\infcat}{\Cat_{\scalebox{1}{$\scriptscriptstyle \infty$}}}
\providecommand{\inftwocat}{\Cat_{\scalebox{1}{$\scriptscriptstyle (\infty,2)$}}}

\providecommand{\spectra}{\cate{S}\text{p}}

\providecommand{\shape}{\sh}
\newsavebox{\pullback}
\sbox\pullback{%
	\begin{tikzpicture}%
		\draw (0,0) -- (1ex,0ex);%
		\draw (1ex,0ex) -- (1ex,1ex);%
\end{tikzpicture}}
\newsavebox{\pushout}
\sbox\pushout{%
	\begin{tikzpicture}%
		\draw (0ex,0ex) -- (0ex,1ex);%
		\draw (0ex,1ex) -- (1ex,1ex);%
\end{tikzpicture}}
\tikzset{%
	symbol/.style={%
		draw=none,
		every to/.append style={%
			edge node={node [sloped, allow upside down, auto=false]{$#1$}}}
	}
}

\let\S\s{}
\let\o\circ

\renewcommand{\Cat}{\mathrm{Cat}}
\newcommand{\fp}{\mathrm{pb}}

\newcommand{\bD}{\mathbb{D}}
\newcommand{\id}{\mathrm{id}}
\newcommand{\fib}{\mathrm{fib}}
\renewcommand{\L}{\mathrm{L}}
\newcommand{\q}{\mathrm{q}}
\newcommand{\Baut}{\mathrm{Baut}}
\newcommand{\aut}{\mathrm{aut}}
\newcommand{\BTop}{\mathrm{BTop}}

\newcommand{\PD}{\mathrm{PD}}

\newcommand{\BHomeo}{\mathrm{BHomeo}}

\theoremstyle{plain}
\newcounter{zaehler}

\newtheorem{introthm}[zaehler]{Theorem}

\newtheorem*{introdef}{Definition}

\title{Homology manifolds via six functor formalisms}

\author[M.~Land]{Markus Land}
\address{JGU Mainz, Institute of Mathematics, Mainz, Germany}
\email{mland@uni-mainz.de}

\author[M.~Volpe]{Marco Volpe}
\address{Universit\"at Regensburg, Universit\"atsstraße 31, 93053 Regensburg, Germany}
\email{Marco.Volpe@mathematik.uni-regensburg.de}

\date{\today}

\begin{document}

\begin{abstract}
We study homology manifolds through the eyes of the six functor formalism of spectral sheaves on locally compact Hausdorff spaces. As main results, we characterize cohomologically smooth objects by adapting an argument of Scholze, deduce that any hypercomplete locally compact ANR homology manifold is cohomologically smooth, show that compact ANR homology manifolds $X$ are Poincar\'e duality complexes whose Spivak tangent fibration identifies with the dualizing sheaf of $X$, and prove a generalization of Wilder's monotone mapping theorem about cell-like maps. Moreover, we introduce the notion of homotopy manifolds for which we prove an unstable analog of Wilder's orientability conjecture and show that hypercomplete ANR homology manifolds are homotopy manifolds. As a consequence, we show that for a compact $d$-dimensional ANR homology manifold, the Spivak tangent fibration of its associated Poincar\'e duality complex canonically destabilizes to a pointed $S^d$-fibration.  Finally, we introduce homotopy manifolds with conical singularities, a generalization of Cohen's triangulated homotopy manifolds, and show that such objects are in fact topological manifolds, generalizing a result of Siebenmann.

Along the way, we obtain comparisons between sheaf and singular cohomology and between the shape and the weak homotopy type of a topological space, explore the relation between various notions of cohomological dimension and hypercompleteness, and study six functor formalisms satisfying the K\"unneth formula.
\end{abstract}
	
\maketitle

\setcounter{tocdepth}{1}
\tableofcontents

\section{Introduction}

\subsection*{Motivation}
Homology manifolds are generalizations of topological manifolds that have been studied since at least the late 1940s, notably by Wilder under the name \textit{generalized manifolds} \cite{Wilder}. Roughly speaking, ($d$-dimensional) homology manifolds are locally compact Hausdorff spaces whose local homology groups at each point look like the ones of a ($d$-dimensional) topological manifold. These local homology groups are the stalks of the so-called \textit{orientation sheaf}, and a big open question raised by Wilder at the time was whether this orientation sheaf is locally constant (or equivalently invertible), a problem solved some 20 years later by Bredon \cite{Bredon} making use of his development of sheaf (co)homology for locally compact Hausdorff spaces. 

Special kinds of homology manifolds, namely those that are also absolute neighborhood retracts (ANRs) of finite covering dimension, have then played a prominent role in the topological surgery classification of manifolds as we explain next. To that end, recall first that for a Poincar\'e duality complex $X$, topological surgery theory can be used to describe the \textit{surgery structure space}, denoted by $\widetilde{\mathcal{S}}(X)$. This appears
classically as an intermediate step in the aim of investigating $\mathcal{S}(X)$, the \textit{manifold structure space} of $X$. The difference between the manifold structure space $\mathcal{S}(X)$ and the surgery structure space $\widetilde{S}(X)$ is still not fully understood. The best general understanding of their difference is via pseudoisotopy theory; in a range of degrees depending linearly on the dimension of $X$, it can be described by (some form of) Waldhausen's A-theory.

Unlike $\mathcal{S}(X)$, the surgery structure space $\widetilde{\mathcal{S}}(X)$ can be fully computed via geometric and algebraic surgery theory of Wall, Kirby--Siebenmann, and Ranicki \cite{Wall, KS, Ranicki-blue} to be an explicit infinite loop space, closely related to the fibre of the assembly map in (quadratic) L-theory. For oriented $X$, there is a map 
\[ \widetilde{\mathcal{S}}(X) \to \fib(X \otimes \L^\q(\Z) \xrightarrow{\mathrm{ass}} \L^\q(X)) \]
which turns out to be an inclusion of path components if $\dim(X) \geq 5$. In particular, since all L-spectra appearing here are 4-periodic, as is the assembly map, this result implies \textit{Siebenmann's periodicity theorem}, asserting that there is a preferred periodicity map
\[ \pi_*(\widetilde{\mathcal{S}}(X)) \to \pi_{*+4}(\widetilde{\mathcal{S}}(X)) \]
which is an isomorphism for $*>0$. But, as indicated above, for $*=0$ this map is a priori only injective, not bijective: this fact is sometimes referred to as \textit{Siebenmann's periodicity mistake}, in that his periodicity theorem does not quite imply that the surgery structure space is actually 4-periodic. However, this periodicity mistake has interesting geometric implications. Elements in $\pi_0(\widetilde{\mathcal{S}}(X))$ determine homeomorphism types of manifolds homotopy equivalent to $X$\footnote{To make this statement true as written, one should use a version of the above with simple homotopy equivalences provided $X$ is equipped with a simple structure.}. It is then natural to ask for the geometric meaning of elements in $\pi_4(\widetilde{\mathcal{S}}(X)))$ that are \emph{not} contained in the image of Siebenmann's periodicity map. Do they correspond to some form of geometric structure on $X$? This question has been coined the \emph{missing manifolds problem} and was solved in the landmark work of Bryant--Ferry--Mio--Weinberger \cite{BFMW}, which makes essential use of homology manifolds. 

Indeed, a compact finite dimensional ANR homology manifold is known to be a Poincar\'e duality complex (we also give a proof of this statement in the body of the text). If one replaces the notion of topological manifolds in the above discussion by such homology manifolds, one may analogously define $\widetilde{\mathcal{S}}^H(X)$, the \textit{homology manifold surgery structure space} of $X$. The main result in \cite{BFMW} identifies $\widetilde{\mathcal{S}}^H(X)$ with the fibre of the assembly map in quadratic L-theory. In particular, the evident map $\widetilde{\mathcal{S}}(X) \to \widetilde{\mathcal{S}}^H(X)$ induces an inclusion on path components, so that homology manifolds may be viewed as the ``missing manifolds''. Moreover, from the assembly map perspective, there is a canonical exact sequence 
\[0 \to \pi_0(\widetilde{\mathcal{S}}(X)) \to \pi_0(\widetilde{\mathcal{S}}^H(X)) \to 1+8\Z. \]
The latter map takes an ANR homology manifold to its Quinn invariant \cite{Quinn1, Quinn2, Quinn3}, which determines whether an ANR homology manifold admits a resolution, that is, a cell-like map from a topological manifold, a fact also shown in \cite{BFMW}. 

The results in \cite{BFMW} mentioned above are based on the assumption that the PD complex $X$ has a topological normal invariant, or equivalently, that the Spivak tangent fibration of $X$ admits a stable euclidean bundle representative. In the case where $X$ is a compact ANR homology manifold, the existence of such a euclidean bundle representative was claimed in \cite{FP}. However, a slight strengthening of this result has been disproven recently \cite{hebestreit2024homology}, see also \cite{BFMW-erratum}, and one believes the result in \cite{FP} to be incorrect; see \cite{hebestreit2024homology} for details. In view of these developments new ideas may be necessary to clarify the role of homology manifolds in surgery theory.

There are still many other interesting open questions about (finite dimensional) ANR homology manifolds. Perhaps the most prominent are the $s$-cobordism theorem and the question whether they are homogenuous, at least in the presence of a mild transversality assumption called the DDP (disjoint disk property). Moreover, many basic results about homology manifolds are not yet treated in modern language, and partly, proofs are not easy to follow (at least for the authors of this paper). We therefore believe that it is desirable to revisit the basic theory of homology manifolds from a modern perspective. In this paper we do so by reinterpreting the theory of homology manifolds through the lens of the six functor formalism on locally compact Hausdorff spaces given by $X \mapsto \Sh{X}{\Sp}$ \cite{volpe2023operations}. As a consequence of a quite general analysis, we obtain for instance the following theorem, summarizing some state-of-the-art about ANR homology manifolds.

\begin{introthm}\label[introthm]{Main-Theorem}
Let $X$ be a $d$-dimensional ANR homology manifold which is $\mathbb{F}_p$-hypercomplete\footnote{That is, $\Sh{X}{\Mod_{\mathbb{F}_p}}$ is hypercomplete.} for all primes $p$. Then 
\begin{enumerate}
\item\label{item:thma1} $X$ is cohomologically smooth in the sense of Scholze. In particular, its dualizing sheaf $\omega_X$ is invertible and hence equivalently described by a map $X \to \Pic(\S{})$.
\item\label{item:thma2} If $X$ is compact, then its underlying homotopy type\footnote{Equivalently, its shape.} is a $d$-dimensional Poincar\'e duality complex whose Spivak tangent fibration $T_X$ canonically identifies with $\omega_X$.
\item\label{item:thma3} If $X$ is hypercomplete, there exists a pointed $S^d$-fibration $\omega_X^{\Ani_{\ast}}\colon X \to \Baut_*(S^d)$ refining the dualizing sheaf $\omega_X$ via composition with $\Baut_*(S^d) \to \Pic(\S{})$. Moreover, $\omega_X^{\Ani_{\ast}}$ is functorial in homeomorphisms, meaning that the canonical map $\BHomeo(X) \to \Baut(X)$ factors through as a composite 
	\[\BHomeo(X) \to \Baut^{\omega_X^{\Ani_*}}(X) \to \Baut(X)\] 
	where $\aut^{\omega_X^{\Ani_*}}(X)$ denotes the anima of automorphims of $\omega^{\Ani_*}_X$ seen as an object in $\Ani_{/\Baut_*(S^d)}$ and the latter map is the forgetful map.
\item if $X$ is hypercomplete and compact, the map $\BHomeo(X) \to \Baut^{\omega_X^{\Ani_*}}(X)$ fits in a commutative square
\[\begin{tikzcd}
	\BHomeo(X) \ar[r] \ar[d] & \Baut^{\omega_X^{\Ani_*}}(X) \ar[d] \\
	\Baut(X) \ar[r] & \Baut^{T_X}(X)
\end{tikzcd}\]
	where similarly as above $\aut^{T_X}(X)$ denotes the anima of automorphisms of the Spivak tangent fibration $T_X$ seen as an object in $\Ani_{/\Pic(\bS)}$.
\end{enumerate}
\end{introthm}

\begin{Rmk}
\begin{enumerate}
\item If $X$ is an ENR, \cref{Main-Theorem}~\eqref{item:thma1} is essentially due to Scholze \cite{scholze2022six}; we offer a precise characterisation of the cohomologically smooth objects below.
\item If $X$ is of finite covering dimension, a common assumption in the literature on ANR homology manifolds, then $X$ is $\mathbb{F}_p$-hypercomplete for all primes $p$. In this situation, the first part of \cref{Main-Theorem}~\eqref{item:thma2} is well-known.
\item Even the existence part of \cref{Main-Theorem}~\eqref{item:thma3} seems to be novel and is orthogonal to the earlier mentioned claim of Ferry--Pedersen that the Spivak tangent fibration admits a refinement to a stable euclidean bundle, that is, that it is classified by a composite $X \to \BTop \to \Pic(\S{})$. 
\item For a topological manifold $X$, there is even a canonical unstable unpointed spherical fibration, i.e.\ $\omega_X^{\Ani_*} \colon X \to \Baut_*(S^d)$ in fact lifts along $\Baut(S^{d-1}) \to \Baut_*(S^d)$, simply because a topological manifold has a topological tangent bundle $TX\colon X \to \BTop(d)$ refining all of the above. While we almost surely know, as explained above, that homology manifolds do not have such euclidean bundles, we do not know whether $\omega_X^{\Ani_*}$ is the fibrewise suspension of a canonical map $X \to \Baut(S^{d-1})$, but we expect that this is not the case. 
\end{enumerate}
\end{Rmk}

\subsection*{More details}
We now give a more technical overview of the contents of this paper. As indicated earlier, our methods lie within the general framework of \textit{six functor formalisms}, which may be viewed as a categorified version of cohomology theories. Roughly speaking, a six functor formalism assigns to each space (or geometric object) a stable $\infty$-category of coefficients together with functorial operations $ f^*, f_*, f_!, f^! $, tensor products, and internal Homs, encoding in a unified way the fundamental functoriality and base-change properties expected of (co)homology in any reasonable geometric context. The main example studied in this paper is the six functor formalism of \textit{sheaves of $R$-modules on locally compact Hausdorff spaces}, where $R$ is a $\mathbb{E}_\infty$-ring. We refer to \cite{volpe2023operations} for a general account of the theory. Other prominent examples in different geometric settings include \textit{\'etale sheaves} \cite{artinGrothendieckVerdierSGA4} and \textit{motivic spectra} \cite{ayoub2007sixopsI} in algebraic geometry.

Recent work by many authors, including \cite{lu2022categorical, scholze2022six, heyer20246, gaitsgory2017study, cnossen2025universality}, has led to a precise formalization of six functor formalisms in terms of \textit{$\infty$-categories of correspondences}. We briefly summarize the approach of \cite{lu2022categorical, scholze2022six, heyer20246}. To any $\infty$-category $\C$ with finite limits, one associates a new $\infty$-category $\Corr(\C)$, called the $\infty$-category of correspondences (or spans). Its objects are those of $\C$, a morphism from $X$ to $Y$ is given by a span
\[
X \xleftarrow{} Z \xrightarrow{} Y,
\]
and composition is given by taking pullbacks.
The cartesian monoidal structure on $\C$ induces a symmetric monoidal structure on $\Corr(\C)$. A six functor formalism on $\C$ is then defined to be a lax symmetric monoidal functor
\[
D \colon \Corr(\C) \to \infcat
\]
such that $D$ sends every morphism to a left adjoint functor and, for each $X \in \C$, the induced symmetric monoidal structure on $D(X)$ is closed. From such a functor $D$, one formally recovers the usual operations of the six functor formalism, including (exceptional) pullbacks and pushforwards, projection formulas, and base change isomorphisms.

In the general setting of a six functor formalism \( D \colon \Corr(\C) \to \infcat \), one can axiomatize when a morphism \( f \) in \( \C \) satisfies Poincar\'e duality with respect to \( D \). This property is known as \textit{cohomological smoothness}, or \textit{\(D\)-smoothness}, and has been studied in \cite{zavyalov2023poincar, scholze2022six}. For the reader's convenience, we briefly recall the relevant definitions. Given a morphism \( f \colon X \to Y \) in \( \C \), we write \( \omega_f \) for its \textit{dualizing object} (or, depending on context, \textit{dualizing sheaf}), defined by
\[
\omega_f := f^!(1_Y) \in D(X),
\]
where \( 1_Y \in D(Y) \) denotes the monoidal unit. When \( \C = \mathrm{LCH} \), the category of locally compact Hausdorff spaces, \( f \colon X \to * \) is the terminal morphism, and \( D = \Sh{-}{\Mod_R} \), we write \( \omega_X^R \) for the corresponding dualizing sheaf.

\begin{introdef}\label{introdefcohsmooth}
    Let $f\colon X\rightarrow Y$ be a morphism in a geometric setup $\C$ and $D$ a six functor formalism on $\C$. We say that $f$ is $D$-\emph{smooth} if for any pullback square in $\C$ as in \eqref{pb} the following conditions hold.
\begin{enumerate}[(i)]
	\item\label{itemdef:i} The map $\pb{v}\omega_f\rightarrow\omega_g$ is invertible.\footnote{Equivalently, $1_X \in D(X)$ is $f$-suave in the sense of \cite[Def.\ 6.1]{scholze2022six}, see \cref{remark:f-suave}.}
	\item\label{itemdef:ii} The object $\omega_f \in D(X)$ is $\otimes$-invertible.
\end{enumerate}
When $\C=\LCH$ and $D=\Sh{-}{\E}$, we say that $f$ is $\E$-smooth rather than $\Sh{-}{\E}$-smooth. If moreover $\E=\Mod_R$ for an $\bE_{\infty}$-ring $R$, we say that $f$ is $R$-smooth.
\end{introdef}

One of the main objectives of this paper is to unpack the notion of $D$-smoothness in the special case where $X$ is a LCH space, $f$ is the unique map $X\rightarrow\ast$, and $D=\Sh{-}{\spectra}$. Much of this introduction is devoted to illustrate our efforts to relate this abstract notion to more classical concepts from geometric topology. 

\subsection*{Local contractibility}
Our first step in this direction is to offer an interpretation of condition \eqref{itemdef:i} above in terms of \textit{local contractibility}. For this we work in the generality of categorical \textit{K\"unneth formulas}. Let $\C$ a $\infty$-category with finite limits, $f$ any morphism in $\C$ and $D\colon \C\op\rightarrow\Cat$ \textit{any} functor (not necessarily extending to a six functor formalism) and write $\pb{f}$ for $D(f)$. 
\begin{enumerate}
\item[-] We say that $D$ satisfies the K\"unneth formula if, for any cartesian square in $\C$, the comparison map $ D(X)\otimes_{D(Z)} D(Y)\rightarrow D(X\times_Z Y)$ is an equivalence. For instance, for any presentable $\infty$-category $\E$, the functor $D=\Sh{-}{\E}$ on LCH satisfies the K\"unneth formula.
\item[-] We say $f$ is \textit{$D$-locally contractible} if $\pb{f}$ admits a left adjoint $\pfs{f}$ satisfying the projection formula. We say $f$ is \textit{universally $D$-locally contractible} if any base change of $f$ is $D$-locally contractible. We say that an object $X$ in $\C$ is (universally) $D$-locally contractible if the unique map $X \to \ast$ is. When $D=\Sh{-}{\E}$, we say that $f$ is $\E$-locally contractible rather than $\Sh{-}{\E}$-locally contractible. Iff moreover $\E=\Mod_R$ for an $\bE_{\infty}$-ring $R$, we say that $f$ is $R$-locally contractible.

\end{enumerate}
In this setup, we provide the following characterization.

\begin{introthm}[\cref{kunneth=>1ULA=loccontr}]\label[introthm]{introthm:1ulaandloccontr}
Let $D$ be a cocomplete six functor formalism on a geometric setup $\C$ which satisfies the K\"unneth formula and $f\colon X \to Y$ a morphism in $\C$.
Then the following assertions are equivalent:
\begin{enumerate}[(i)]
	\item\label{item:i} $f$ is $D$-locally contractible
	\item\label{item:ii} the monoidal unit ${1}_X\in D(X)$ is $f$-suave
	\item\label{item:iii} $f$ is $D$-universally locally contractible.
\end{enumerate}
\end{introthm}

\begin{Rmk}
\begin{enumerate}
\item In practice, it is easier to check that a map $f$ is $D$-local contractible (e.g.\ using adjoint functor theorem in presentable situations) than to check that $1_X$ is $f$-suave, but it is much more useful to know that $1_X$ is $f$-suave. Note also that the notion of $D$-local contractibility does not require $D$ to be part of a six functor formalism.
\item For $\C$ the category of topological spaces, $D$-local contractibility is quite close to familiar concepts from topology. For instance, $X$ is $\Ani$-locally contractible if and only if $X$ is \textit{locally of constant shape} in the sense of Lurie \cite[Appendix A]{lurie2017higher}. Indeed, recall that for $f\colon X \to \ast$, the functor $\pb{f}$ admits a pro-left adjoint $\Sh{X}{\Ani} \to \Pro(\Ani)$ sending the terminal sheaf to the \textit{shape} of $X$, denoted by $\sh(X)$. If $X$ is $\Ani$-locally contractible, we find that this pro-left adjoint is given by the composite of the left adjoint $\pfs{f}$ with the inclusion $\Ani \subseteq \Pro(\Ani)$ as constant pro objects. The same remains true for any open subset of $X$, so that $X$ is locally of constant shape. 
\item Still for $\C$ the category of topological spaces, and for a field $K$, requiring $X$ to be $K$-locally contractible is closely related to $X$ being \textit{cohomologically locally $\infty$-connected} as studied classically by Wilder, Bredon, and many others, see \cref{lemma:clcooandloccontrK} for details. 
\end{enumerate}
\end{Rmk}

A reader coming from point-set topology is probably more familiar with the following notion of local weak contractibility of a space $X$. Namely, the condition that for each point $x\in X$ and open $U\subseteq X$ with $x\in U$, there is a neighbourhood $V\subseteq U$ of $x$ and a homotopy between $\Sing(V)\rightarrow\Sing(U)$ and $\Sing(V)\rightarrow\Sing(\{x\})\rightarrow\Sing(U)$. Here, $\Sing(X)\in\Ani$ denotes the weak homotopy type of a topological space $X$. For a complete $\infty$-category $\D$, one may define similarly the notion of \textit{$\D$-local weak contractibility} using cotensors with the anima $\Sing(U)$ for each open $U$ in $X$, see \cref{singcontractibility} for a precise definition. In general, local weak contractibility does not imply $\Ani$-local contractibility, as we argue in \cref{rmk:counterexampleforlocwc=>shvloccontr}. Nevertheless, we prove the following result.

\begin{introthm}[\cref{equivsheafcohhom}, \cref{equivshandsing}]
Let $\D$ be $\text{Mod}_R$ with $R$ a connective $\bE_{\infty}$-ring spectrum and $X$ be a $\D$-locally weakly contractible space. Then $X$ is $\D$-locally contractible if and only if all $\D$-valued constant sheaves are hypercomplete. In this case, we have isomorphisms
\begin{align*}
	\sing{X}\otimes M\simeq\pfs{a}^{\D}\pb{a}_{\D}M  \quad \text{ and } \quad  M^{\sing{X}} \simeq\pf{a}^{\D}\pb{a}_{\D}M
\end{align*}
for any $M\in\D$. The same result as above holds when $\D=\Ani$ and $X$ is additionally required to be first countable. When $X$ is LCH and $\D = \Mod_R$, we furthermore have
\[\sing{X}\otimes M\simeq\pfp{a}^{\D}\pbp{a}_{\D}M.\]
\end{introthm}

\begin{Rmk}
For $\Ani$-locally contractible spaces $X$, the displayed equivalences in particular identify the weak homotopy type $\sing{X}$ with the shape $\sh(X)$ and singular (co)homology with sheaf (co)homology with arbitrary coefficients $M$. In case $M$ is a bounded above spectrum, for instance the Eilenberg--Mac Lane spectrum associated to an abelian group, the equivalence $M^{\sing{X}} \simeq \pf{a}^{\Sp}\pb{a}_{\Sp}M$ holds without the assumption that constant sheaves are hypercomplete, see \cref{sheafvssingcohwithboundedcoeff}, \cref{Clocwcontr=>sheafvssinginftyconn}.
\end{Rmk}

For a $\Ani$-locally contractible space $X$, Lurie gives an elegant account of the \textit{monodromy equivalence}, identifying locally constant sheaves on $X$ and spectra parametrised over the shape $\sh(X)$ of $X$, see \cite[Appendix A.1]{lurie2017higher}. We make use of this monodromy equivalence and the well-known characterization of Poincar\'e duality complexes in terms of parametrized spectra summarized e.g.\ in \cite{Land-MJM} to show:

\begin{introthm}[\cref{thm:loccontrlocconstdualisPD}]
Let $X$ be a compact Hausdorff space and $R$ a connective $\bE_\infty$-ring spectrum. Assume that $X$ is $\Ani$-locally contractible and that $\omega_X^R$ is locally constant.
Then $\omega_X^R$ is in fact invertible, $D_{\sh(X)}^R = (\omega_X^R)^{-1}$, and in particular, $\sh(X)$ is $R$-Poincar\'e duality complex.
\end{introthm}
We note that this theorem in particular applies to $\Ani$-locally contractible and $R$-smooth CH spaces $X$.

\subsection*{$R$-homology manifolds}
We now explain how condition \eqref{itemdef:ii} in the definition of $D$-smoothness, i.e.\ the invertibility of $\omega_f$, relates to homology manifolds. The first observation is that when $X$ is $R$-locally contractible, we have the following computation for the stalks of the dualizing sheaf
\[ (\omega_X^R)_x = \text{cofib}(\sh(X\setminus\{x\})\rightarrow\sh(X)),\]
see \cref{lemma:stalkdualizingshapeatx}. Motivated from the classical definition of homology manifolds, 
we call a LCH space a \textit{$R$-homology manifold} if all stalks of $\omega_X^R$ are isomorphic to shifts of $R$. When $R$ is connective and $\pi_0(R)$ is a PID (e.g. $R=\s{}$), this is equivalent to requiring such stalks to be invertible. The following theorem shows that, under some mild additional assumptions, invertibility on stalks is sufficient for $\omega_X^R$ to be invertible. This generalizes Bredon's proof of Wilder's local orientability conjecture.

\begin{introthm}[\cref{prop:locally-constant}]\label[introthm]{introthm:cohsmoothchar}
Let $R$ be a connective $\bE_\infty$-ring and $X$ a $R$-locally contractible LCH space. Then $X$ is $R$-smooth if and only if the following conditions hold. 
\begin{enumerate}
    \item For every maximal ideal $\mathfrak{m} \subseteq \pi_0(R)$, $\omega_X^K$ is hypercomplete, where $K=\pi_0(R)/\mathfrak{m}$;
    \item For all $x \in X$, the stalk $x^*(\omega_X^{\pi_0(R)})$ is invertible;
    \item For all $y,z\in X$ lying in the same connected component, we have $y^*(\omega_X^R)\simeq z^*(\omega_X^R)$.
\end{enumerate}
\end{introthm}

Thanks to \cref{introthm:1ulaandloccontr}, the proof of \cref{introthm:cohsmoothchar} amounts to showing that the listed conditions are equivalent to the invertibility of $\omega_X^R$. To do so, we adapt an argument of Scholze \cite[Prop.\ 7.9]{scholze2022six}. When $K$ is a field, we further show that a LCH space $X$ is $K$-smooth if and only if it is a generalized manifold in the sense of Wilder. This is an immediate consequence of considerations on the \textit{$!$-cohomological dimension}, as defined in \cref{section:cohdim}. We refer to \cref{rmk:wilder=assumptioninprop} and \cref{cohsmoothcharPID} for details.

A useful consequence of \cref{introthm:cohsmoothchar} is the stability under \textit{cell-like maps} of cohomological smoothness. The next result is a modern reformulation and a generalization of \textit{Wilder's monotone mapping theorem} (see \cite[Theorem 16.33]{bredon2012sheaf}). 

\begin{introthm}[\cref{thm:wildermonotone}]
Let $R$ be a connective $\bE_{\infty}$-ring, and let $f\colon X\rightarrow Y$ be an $R$-cell-like map, where $X$ is $R$-smooth. Suppose that there exists a $\otimes$-invertible $R$-module $M$ such that one of the following two conditions hold
	\begin{enumerate}
		\item $\omega_X^R \simeq M_X$ is equivalent to the constant sheaf on $M$, or
		\item $\pi_0(R)\cong\mathbb{Z}$ and all stalks of $\omega_X^{\mathbb{Z}}$ are isomorphic to $M\otimes\mathbb{Z}$.
	\end{enumerate}
    Then $Y$ is $R$-smooth.
\end{introthm}

\subsection*{$\Ani_*$-homotopy manifolds}
We next consider the following special kinds of homology manifolds, that we call \textit{$\Ani_\ast$-homotopy manifolds}, or sometimes just homotopy manifolds for short. These are $\Ani$-locally contractible LCH spaces $X$ with the property that, for each point $x\in X$, we have an equivalence
\[\mathrm{cofib}(\sh(X\setminus\{x\})\rightarrow\sh(X)) \simeq S^n\] 
for some $n\geq0$. We note that topological manifolds are natural examples of $\Ani_*$-homotopy manifolds, and warn the reader that this notion differs from that of homotopy manifolds introduced by Griffiths \cite{Griffiths}, see \cref{rmk:homotopymannotgriffiths}. We then show that in sufficiently favorable situations, $\Ani_\ast$-homotopy manifolds are essentially the same as $\spectra$-homology manifolds. 

\begin{introthm}[\cref{homology=>htpyman}]\label[introthm]{introthm:homology=>htpyman}
    Let $X$ be a LCH space which is $\Ani$-locally contractible. Assume that $X$ is a $\spectra$-homology manifold of dimension $>1$. Then $X$ is a $\Ani_*$-homotopy manifold.
\end{introthm}
For instance, an ANR homology manifold as classically defined is an $\Ani$-locally weakly contractible $\Sp$-homology manifold, and therefore by \cref{introthm:homology=>htpyman} a $\Ani_*$-homotopy manifold. 

For $\Ani_\ast$-homotopy manifolds, we then provide an unstable counterpart of \cref{introthm:cohsmoothchar}. A crucial ingredient is that for any $\Ani$-locally contractible space $X$, the dualizing sheaf $\omega_X$ canonically destabilizes to a sheaf of pointed anima $\omega_X^{\Ani_\ast}$, see \cref{lemma:destabdualsheaf}. The unstable analogue of Wilder's local orientability conjecture is then:

\begin{introthm}[\cref{thm:unstablelocalorient},\cref{thm:functorialunstabledual}]\label[introthm]{introthm:unstablelocalorient}
    Let $X$ be a hypercomplete $\Ani_*$-homotopy manifold of dimension $d$. Then 
    \begin{enumerate}
    \item $\omega_X^{\Ani_{\ast}}$ is locally constant, that is, a pointed spherical fibration over $\sh(X)$. Moreover, $\omega_X^{\Ani_{\ast}}$ is functorial in homeomorphisms, meaning there is a canonical map $\BHomeo(X) \to \Baut^{\omega_X^{\Ani_*}}(X)$, where $\Baut^{\omega_X^{\Ani_*}}(X)$ denotes the full subgroupoid of $\Ani_{/\Baut_*(S^d)} \subseteq \Ani_{/\Ani_*}$ containing the object $(\sh(X),\omega_X^{\Ani_*})$, lifting the canonical map $\BHomeo(X) \to \Baut(X)$. 
    \item If $X$ is in addition compact, the just described map participates in a commutative diagram
    	\[\begin{tikzcd}
	\BHomeo(X) \ar[r] \ar[d] & \Baut^{\omega_X^{\Ani_*}}(X) \ar[d] \\
	\Baut(X) \ar[r] & \Baut^{T_X}(X).
\end{tikzcd}\]
    \end{enumerate}

\end{introthm}

Combining \cref{introthm:homology=>htpyman} and \cref{introthm:unstablelocalorient}, we obtain the canonical destabilization of the Spivak tangent fibration of hypercomplete compact ANR homology manifolds described in \cref{Main-Theorem}\eqref{item:thma3}.

Finally, we consider more geometric examples of homotopy manifolds, that we call \textit{homotopy manifolds with conical singularities}, see \cref{def:conical-singularities}. This is a natural generalization of the notion of homotopy manifolds studied among others by Siebenmann and Cohen \cite{siebenmann1970non,cohen1970homeomorphisms}, which replaces triangulations with well-behaved stratifications. We then use the topological Poincar\'e conjecture to prove that homotopy manifolds with conical singularities are in fact topological manifolds, generalizing a theorem of Siebenmann.

\begin{introthm}[\cref{thm:htpymanwithconeareman}]
    Let $X\rightarrow P$ be a homotopy manifold with conical singularities. Then $X$ is a topological manifold.
\end{introthm}

\subsection{Acknowledgements}
MV thanks Ko Aoki for pointing out \cite{cohen1970homeomorphisms} and Walter Tholen for lending his copy of Engelking's textbook \cite{Engelking}. ML thanks Peter Scholze for correspondence about local to global Poincar\'e duality at very early stages of this project.
	 	
\section{Recollections on six functor formalisms}

\subsection{Correspondences}

This section serves to collect the basic properties and definitions involved in general six functor formalisms, mainly following \cite{scholze2022six}, and expanding on Scholze's notes. In particular, we record a proof of the technically important \cref{functorialcohcorr}, which Scholze takes as given. While we were in the process of completing our paper, \cite{heyer20246} appeared, which contains a proof of \cref{functorialcohcorr} as well as many foundational results on $(\infty,2)$-categories of cohomological correspondences.

We start by recalling some properties of the \emph{$(\infty, 1)$-category of correspondences} $\Corr(\C)$ associated with an $\infty$-category $\C$ which admits pullbacks. We refer to \cite[Section 2]{haugseng2020two} for a modern treatment of this construction, called the $\infty$-category of spans in loc.\ cit., but see also \cite{barwick2017spectral}, \cite{Bachmann2017NormsIM}. In \cite[Section 2]{haugseng2020two}, the authors consider a more general construction of a category of correspondences associated to what they call an \emph{adequate triple}. Any $\infty$-category with pullbacks naturally defines an adequate triple \cite[Example 2.3 (2)]{haugseng2020two}. For our applications, working with $\infty$-categories with pullbacks is sufficient, so we will not recall the general definition of adequate triples here.

Let us denote by $\infcat^{\fp}$ the subcategory of $\infcat$ whose objects are $\infty$-categories with pullbacks and whose morphisms are functors which preserve pullbacks. By the discussion preceding \cite[Proposition C.20]{Bachmann2017NormsIM}, $\infcat^{\fp}$ can be equipped with the structure of an $(\infty,2)$-category with mapping $(\infty,1)$-categories $\Hom{\Cat^\fp_\infty}{\C}{\D}$ given by the subcategory of $\Fun(\C,\D)$ on pullback preserving functors and cartesian transformations between them; For $F, G \colon \C \to \D$ two functors and a natural transformation $\alpha\colon F\Rightarrow G$, recall that $\alpha$ is said to be \emph{cartesian} if for any morphism $X\rightarrow Y$ in $\C$, the square 
\[\begin{tikzcd}
	F(X) \ar[d, "\alpha_X"'] \ar[r] & F(Y) \ar[d, "\alpha_Y"] \\
	G(X) \ar[r] & G(Y)                      
\end{tikzcd}\]
is cartesian in $\D$. An adjunction is called a \emph{cartesian adjunction} if unit and counit are cartesian transformations. The cartesian adjunctions are then precisely the internal adjunctions of the $(\infty,2)$-category $\Cat_\infty^\fp$.

\begin{example}\label[example]{loccartclcartadjunction}
Let $\C$ be a an $\infty$-category with pullbacks, and let $f\colon X\to Y$ be a morphism in $\C$.  Then the adjunction 
\[\begin{tikzcd}
    	\C_{/X} \ar[r,bend left,"{f_{\o}}",""{name=A, below}] & \C_{/Y},\ar[l,bend left,"{\pb{f}}",""{name=B,above}] \ar[from=A, to=B, symbol = \dashv]
\end{tikzcd}\]
where $f_{\o}$ denotes post-composition with $f$ and $\pb{f}$ denotes pullback along $f$, is cartesian.
\end{example}

\begin{proposition}\label[proposition]{prop:corr}
\begin{enumerate}
\item\label{corrmon} Associating to an $\infty$-category with pullbacks $\C$ its $\infty$-category of correspondences $\Corr(\C)$ refines to a 2-functor $\Corr \colon \Cat_\infty^\fp \to \Cat_\infty$. The underlying functor of $(\infty,1)$-categories preserves all small limits. In particular, if $\C$ is in $\Cat_\infty^\fp$ and admits a terminal object (and hence finite products), then $\Corr(\C)$ admits a canonical symmetric monoidal structure, which on underlying objects is the product in $\C$.
\item\label{corrselfdual} There is a natural equivalence of functors $\Corr(-) \simeq \Corr(-)\op$.
\end{enumerate}
\end{proposition}
\begin{proof}
\eqref{corrmon} follows from \cite[Lemma 2.4 \& Theorem 2.18]{haugseng2020two} and \cite[Proposition C.20]{Bachmann2017NormsIM} and \eqref{corrselfdual} from \cite[Lemma 2.14]{haugseng2020two}.
\end{proof}

\begin{corollary}\label[corollary]{corrpreservesadjunctions}
Let $\C$ and $\D$ be $\infty$-categories with pullbacks and $L\colon \C\to \D$ and $R\colon \D\rightarrow \C$ be pullback preserving left and right adjoints of a cartesian adjunction. Then $\Corr(L)$ is left and right adjoint to $\Corr(R)$. 
\end{corollary}
\begin{proof}
By \cref{prop:corr}\eqref{corrselfdual}, it suffices to prove that $\Corr(L)$ is left adjoint to $\Corr(R)$. By \cref{prop:corr}\eqref{corrmon} $\Corr$ can be promoted to a $2$-functor and consequently preserves adjunctions as needed.
\end{proof}

\begin{corollary}\label[corollary]{corrloccartadj}
Let $\C$ be an $\infty$-category with pullbacks, and let $f\colon X\rightarrow Y$ be a morphism in $\C$. Then $\Corr(f_{\o})\colon \Corr(\C_{/X})\rightarrow\Corr(\C_{/Y})$ is both left and right adjoint to $\Corr(\pb{f})$.
\end{corollary}
\begin{proof}
It remains only to record that for all  objects $Z$, the slice category $\C_{/Z}$ admits pullbacks which are formed underlying, and that consequently $\pfs{f}$ preserves pullbacks.
\end{proof}

For the remainder of this section we will fix a category $\C$ with finite limits, thought of as a geometric setup $(\C,\C)$ in the sense of Scholze \cite{scholze2022six}; i.e.\ we assume that the class of morphisms $E$ in loc.\ cit.\  consists of all maps in $\C$.\footnote{Indeed, we will often use that $\Corr(\C) \simeq \Corr(\C)\op$, which is not in general the case if $E \neq \C$.} 
\begin{terminology}\label[terminology]{terminology}
A six functor formalism on a geometric setup $\C$ consists of a functor
\[D \colon \Corr(\C) \rightarrow \Cat_\infty\]
satisfying various axioms. First, $D$ is required to be lax symmetric monoidal, i.e.\ a  three functor formalism in the sense of \cite{scholze2022six}. Moreover, we require all operations to admit right adjoints. More precisely, for all $X$ in $\C$, one requires that the symmetric monoidal structure $\otimes$ on $D(X)$ is closed with internal hom objects $\sHom{D(X)}{-}{-}$, and for all maps $f\colon X \to Y$ in $\C$, one requires the functors $f_!\colon D(X) \to D(Y)$ and $f^*\colon D(Y) \to D(X)$ to admit right adjoints $f^!$ and $f_*$. 
For various examples of six functor formalisms appearing in practice, we refer to the lecture notes of Scholze \cite{scholze2022six}.

For a three functor formalism $D$ on a geometric context $\C$, one obtains two formulas which are used throughout this paper: basechange and the projection formula. To state them, 
consider a pullback square in $\C$ of the form
\begin{equation}\tag{$\square$}\label{pb}
\begin{tikzcd}
	W \ar[r,"v"] \ar[d,"g"'] & X \arrow[d,"f"] \\
	Z \ar[r,"u"] & Y,               
\end{tikzcd}
\end{equation}
Then there is a canonical equivalence of functors 
\[ u^*f_! \simeq g_!v^* \colon D(X) \to D(Z)\]
called \emph{base-change}. The \emph{projection formula} refers to a canonical equivalence of functors
\[ f_!(-)\otimes - \simeq f_!(- \otimes f^*(-)) \colon D(X) \times D(Y) \to D(Y).\]

These equivalences furthermore induce the following natural maps:
\begin{enumerate}
\item\label{twisting-upper-shriek} The map $\omega_g\otimes\pb{g}(-)\rightarrow \pbp{g}(-)$ adjoint to 
	\[\pfp{g}(\omega_g\otimes\pb{g}(-))\simeq\pfp{g}\pbp{g}(1_Z)\otimes(-)\xrightarrow{\text{counit}}(-).\]
\item\label{pullback-dualizing-object} The map $\pb{v}\omega_f\rightarrow\omega_g$ adjoint to 
	\[\pfp{g}\pb{v}\pbp{f}(1_Y)\simeq\pb{u}\pfp{f}\pbp{f}(1_Y)\xrightarrow{\text{counit}}\pb{u}(1_Y).\]
\item\label{comparison-upper-shriek-upper-star} The map $v^*f^! \to g^*u^!$ adjoint to
	\[ f^! \to f^!u_*u^* \simeq v_*g^!u^* \]
	where the equivalence is obtained from the base-change equivalence by passing to right adjoints.
\end{enumerate}
\end{terminology}

Given a three functor formalism $D$ on a geometric setup $\C$, we now recall the definition of the $(\infty,2)$-category $\cohcorr{D}(\C)$ of \emph{$D$-based correspondences}. Its homotopy $2$-category originally appeared in \cite[IV.2.3.3]{fargues2021geometrization}, as a variation of \cite[Construction 2.6]{lu2022categorical} as well as \cite[3.2.8]{cisinski2021cohomological}. The precise construction requires some technical results from the theory of enriched $\infty$-categories, that were proven in \cite{gepner2015enriched}. We refer to \cite[Section 2.2]{zavyalov2023poincar} for details. 

\begin{construction}\label[construction]{constructionofcohcorr}
Recall that every object in the monoidal $\infty$-category $(\Corr(\C), \times)$ is dualizable and self-dual, see e.g.\ \cite[Lemma 2.2.5]{zavyalov2023poincar}. Therefore, it is a closed monoidal $\infty$-category, with internal hom between two objects $X$ and $Y$ in $\Corr(\C)$ given by $X\times Y$. As a consequence, $\Corr(\C)$ is canonically enriched over itself. We can then use the lax symmetric monoidal functor $D\colon (\Corr(\C), \times)\rightarrow(\infcat, \times)$ and obtain an $\Cat_\infty$-enrichment of $\Corr(\C)$.
\end{construction}
In this paper, we think of $(\infty,2)$-categories as $\Cat_\infty$-enriched $\infty$-categories. Hence we define:
\begin{definition}\label{defcohcorr}
We define $(\infty,2)$-category $\cohcorr{D}(\C)$ of \textit{$D$-based correspondences} to be the $(\infty,2)$-category obtained from \cref{constructionofcohcorr}.
\end{definition}

Informally, $\cohcorr{D}(\C)$ is described as follows. The objects in $\cohcorr{D}(\C)$ are the objects of $\C$. For any two objects $X$ and $Y$ in $C$, the $\infty$-category of morphisms between $X$ and $Y$ is given by $D(X\times Y)$. For $X$, $Y$ and $Z$ objects of $\cohcorr{D}(\C)$, the composition functor is given by the composite
\[ D(X\times Y) \times D(Y\times Z) \xrightarrow{\otimes} D(X \times Y \times Y \times Z) \xrightarrow{\pfp{{p_{X,Z}}}} D(X\times Z).\]

\begin{remark}\label[remark]{selfdualcohcorr}
   Note that there is an equivalence of $(\infty,2)$-categories $\cohcorr{D}(\C)\op\simeq\cohcorr{D}(\C)$. A full proof of this claim can be found in \cite[Proposition 4.1.4]{heyer20246}. Very roughly, the claim follows from the fact that the equivalence in \cref{prop:corr} can be upgraded to an equivalence of $\Corr(\C)$-enriched $\infty$-categories (see \cite[Corollary 2.4.2]{heyer20246}), and the fact that transporting the enrichment commutes with taking opposites (see \cite[Lemma C.3.4]{heyer20246}).
\end{remark}

\begin{remark}
Note that for a geometric setup $\C$ and an object $S$ in $\C$, the canonical functor $\C_{/S} \to \C$ preserves pullbacks and is lax symmetric monoidal with respect to cartesian products. In particular, it induces a lax symmetric monoidal functor $\Corr(\C_{/S}) \to \Corr(\C)$ which, together with a three functor formalism $D \colon \Corr(\C) \to \Cat_\infty$ shows that $D$ is naturally also a three functor formalism on $\C_{/S}$, we call this the \emph{induced} three functor formalism.
If $D$ is in fact a six functor formalism on $\C$, then the induced three functor formalism on $\C_{/S}$ is again a six functor formalism.
In the next proposition we give a precise formulation and a proof of \cite[Remark 6.2]{scholze2022six} about the coherence of the formation of $D$-based correspondences of slice categories.
\end{remark}

\begin{proposition}\label[proposition]{functorialcohcorr}
Let $D$ be a three functor formalism on a geometric setup $\C$. Then there exists a canonical functor $\cohcorr{D}\colon \C\op \to \Cat_{(\infty,2)}$ with the following properties: It sends an object $S$ in $\C$ to the $(\infty,2)$-category $\cohcorr{D}(\C_{/S})$, and a morphism $f\colon S' \to S$ to a 2-functor $\cohcorr{D}(\C_{/S}) \to \cohcorr{D}(\C_{/S'})$ informally given by sending an object $X$ in $\C_{/S}$ to $X'=\pb{f}(X)$ in $\C_{/S'}$, and for a pair objects $X,Y \to S$ of objects in $\C_{/S}$, associating the functor 
\[ \pb{g} \colon D(X \times_S Y) \to D(X' \times_{S'} Y') \]
where $g \colon X' \times_{S'} Y' \to X \times_S Y$ is the canonical morphism.
\end{proposition}
\begin{proof}
Since $\C$ has pullbacks, there is a functor $\C\op\rightarrow\Cat_\infty^\fp$ which assigns to each object $S$ of $\C$ the slice category $\C_{/S}$, and to each morphism $f\colon S'\rightarrow S$ the basechange along $f$ functor $\pb{f}$. Notice that $\pb{f}\colon \C_{/S}\rightarrow \C_{/S'}$ is a symmetric monoidal functor, when both sides are equipped with the respective cartesian monoidal structures. Therefore, the above upgrades to a functor $\C\op\rightarrow\text{CAlg}_{\C}(\Cat_\infty^\fp)$.
    Since $\Corr$ preserves products by \cref{prop:corr}\eqref{corrmon}, postcomposing with $\Corr$ gives a functor $$\C\op\rightarrow\text{CAlg}_{\Corr(\C)}(\infcat).$$
    By \cref{corrloccartadj}, for each morphism $f$ in $\C$, $\Corr(\pb{f})$ is a left adjoint. Therefore, for each $S\in \C$, $\Corr(\C_{/S})$ admits a closed tensoring over $\Corr(\C)$ induced by basechange along the unique map $S\rightarrow \ast$. By \cite[Theorem 1.1, Proposition 6.10 (5)]{heine2023equivalence}, an $\infty$-category with a closed tensoring over $\Corr(\C)$ can be naturally equipped with a $\Corr(\C)$-enrichment. Hence, we obtain a functor $\C\op\rightarrow\Cat_\infty^{\Corr(C)}$.
Again by transporting the enrichment along the lax symmetric monoidal functor $D\colon\Corr(\C)\rightarrow\Cat_\infty$, one finally obtains a functor $\C\op\rightarrow\inftwocat$. Unravelling the definitions, this functor has all the claimed properties.
\end{proof}

\subsection{Suave objects}
Let $f\colon X\rightarrow Y$ be a map in a geometric setup $\C$ and note that $D(X) = D(X\times_Y Y) = \Hom{\cohcorr{D}(\C_{/Y})}{X}{Y} = \Hom{\cohcorr{D}(\C_{/Y})}{Y}{X}$, so that objects of $D(X)$ may be thought of as morphisms in the $(\infty,2)$-category $\cohcorr{D}(\C_{/Y})$ from $f$ to $\id_Y$ and from $\id_Y$ to $f$. The following definition can be found in \cite[Def.\ 6.1]{scholze2022six}.
\begin{definition}\label[definition]{ULA}%
An object $F$ in $D(X)$ is called \textit{$f$-suave} if $F$ is a left adjoint when viewed as a morphism in the $(\infty,2)$-category $\cohcorr{D}(\C_{/Y})$.
\end{definition}

\begin{lemma}\label[lemma]{ULAstabilityunderbasechange}
Consider a pullback square in $\C$ as in \eqref{pb}
and assume that $F\in D(X)$ is $f$-suave. Then $\pb{v}F\in D(W)$ is $g$-suave.
\end{lemma}
\begin{proof}
This follows from \cref{functorialcohcorr} and the fact that $2$-functors preserve adjunctions.
\end{proof}

\begin{definition}
Let $f\colon X\rightarrow Y$ be a map in a geometric setup $\C$ and $D$ a six functor formalism on $\C$. We denote by $1_Y$ the tensor unit of the symmetric monoidal category $D(Y)$ and by $\omega_f \in D(X)$ its exceptional pullback $f^!(1_Y)$. We call $\omega_f$ the \emph{relative dualizing object}. In case $Y$ is a terminal object of $\C$, we write $\omega_X$ instead of $\omega_f$. We write $\bD_f(-)$ for the functor $\sHom{D(X)}{-}{\omega_f} \colon D(X)\op \to D(X)$ and call $\bD_f(F)$ the \emph{relative Verdier dual} of $F$. 
\end{definition}

\begin{proposition}\label[proposition]{propULA}
    Let $F\in D(X)$ be $f$-suave with right adjoint $G\in D(X)$.
    \begin{enumerate}[(i)]
        \item The object $G\in D(X)$ is $f$-suave with right adjoint $F$.
        \item There is a natural isomorphism
        \[G\otimes\pb{f}(-)\simeq \sHom{D(X)}{F}{\pbp{f}(-)}.\]
        In particular, $G\simeq\bD_f(F)$.
        \item The canonical map $F \to \bD_f(\bD_f(F))$ is an equivalence.
        \item For any pullback square as in \eqref{pb}, we have an isomorphism $\pb{v}\bD_f(F)\xrightarrow{\simeq}\bD_g(\pb{v}F)$.
    \end{enumerate}
\end{proposition}
\begin{proof}
    Point (i) follows from \cref{selfdualcohcorr}. Points (ii) and (iii) are explained in detail in \cite[Proposition 6.5]{scholze2022six}. Point (iv) follows from (ii), \cref{ULAstabilityunderbasechange} and the fact that $2$-functors preserve adjunctions.
\end{proof}

Unravelling the definition of the map in \cref{propULA} (iv) in the special case $F=1_Y$, we obtain the canonical map $v^*\omega_f \to \omega_g$ described in \cref{terminology}. 
In particular, we obtain the following.
\begin{corollary}\label[corollary]{unitULAdualizingsheafstableunderbc}
    Let $f\colon X\rightarrow Y$ be a map in the geometric setup $\C$, and assume that $1_X\in D(X)$ is $f$-suave. Then, for any pullback square in $\C$ of the form \eqref{pb}, the map  $v^*\omega_f \to \omega_g$ is an isomorphism.
\end{corollary}

Conversely, one can also check that objects are $f$-suave as follows. Consider the pullback square 
\begin{equation}\label{pbsquareULA}
    \begin{tikzcd}
X\times_Y X \arrow[d, "p_2"'] \arrow[r, "p_1"] & X \arrow[d, "f"] \\
X \arrow[r, "f"']                              & Y.               
\end{tikzcd}
\end{equation}
\cite[Proposition 6.6]{scholze2022six} then states:

\begin{proposition}\label[proposition]{charULA}\label[proposition]{charunitULA}
The object $F \in D(X)$ is $f$-suave if and only if the natural map
\[ \pb{p_1}F \otimes \pb{p_2}\bD_f(F) \to \sHom{D(X\times_Y X)}{\pb{p_2}(F)}{\pbp{p_1}F}\]
is an equivalence.
\end{proposition}

\begin{remark}\label[remark]{remark:f-suave}%
In particular, the object $1_X\in D(X)$ is $f$-suave if and only if the canonical map
$\pb{p_2}\omega_f \rightarrow\omega_{p_1}$ is an equivalence, and hence if and only if for any pullback square in $\C$ of the form \eqref{pb}, the map $v^*\omega_f \to \omega_g$ is an equivalence. Furthermore, if $1_X$ is $f$-suave, the canonical map $1_X \to \sHom{D(X)}{\omega_f}{\omega_f}$ is invertible.
\end{remark}

\subsection{$D$-smoothness}
Next, we recall the notion of cohomologically smooth maps in the sense of Scholze. We refer to this notion as $D$-smooth maps.
\begin{definition}\label[definition]{defcohsmooth}
    Let $f\colon X\rightarrow Y$ be a morphism in a geometric setup $\C$ and $D$ a six functor formalism on $\C$. We say that $f$ is $D$-\emph{smooth} if for any pullback square in $\C$ as in \eqref{pb} the following conditions hold.
\begin{enumerate}[(i)]
	\item The map $\pb{v}\omega_f\rightarrow\omega_g$ is invertible; equivalently, $1_X$ is $f$-suave.
	\item The object $\omega_f \in D(X)$ is $\otimes$-invertible.
\end{enumerate}
When $\C=\LCH$ and $D=\Sh{-}{\E}$, we say that $f$ is $\E$-smooth rather than $\Sh{-}{\E}$-smooth. If moreover $\E=\Mod_R$ for an $\bE_{\infty}$-ring $R$, we say that $f$ is $R$-smooth.
\end{definition}

\begin{remark}
Consider again a pullback square in $\C$ of the form \eqref{pb}.
If $1_X$ is $f$-suave, then $1_W$ is $g$-suave, see \cref{ULAstabilityunderbasechange}. In particular, applying \cref{propULA} part (ii), with $F=1_W$ and $g$ in place of $f$, we obtain that the canonical map $\omega_g \otimes g^* \to g^!$ is an equivalence. In particular, $D$-smooth maps are precisely the cohomologically smooth maps in the sense of Scholze, see \cite[Definition 5.1]{scholze2022six}.
\end{remark}

\begin{remark}\label[remark]{rmk:invertible-vs-dualizable}
In the definition of $D$-smooth maps, we could also replace the condition that $\omega_f$ is invertible with the condition that $\omega_f$ is dualizable. Indeed, recall that in a closed symmetric monoidal category, an object $X$ is dualizable if and only if for all objects $Y$ the canonical map $\sHom{}{X}{1} \otimes Y \to \sHom{}{X}{Y}$ is invertible. Hence, if $\omega_f$ is dualizable and $1_X$ is $f$-suave, then $1 \simeq \sHom{}{\omega_f}{\omega_f} \simeq \sHom{}{\omega_f}{1} \otimes \omega_f$ showing that $\omega_f$ is in fact invertible.
\end{remark}

\subsection{Local contractibility}

Let $\infcat^L$ be the $\infty$-category whose objects are cocomplete $\infty$-categories, and morphisms are cocontinuous functors. Recall that this carries a symmetric monoidal structure (see \cite[Corollary 4.8.1.4]{lurie2017higher}) that we simply denote by $\otimes$. We write $\text{CAlg}(\infcat^L)$ for the $\infty$-category of commutative algebras in $\infcat^L$ with respect to $\otimes$. Objects in $\text{CAlg}(\infcat^L)$ are symmetric monoidal $\infty$-categories whose tensor product preserves colimits in both variables and morphisms are colimit preserving symmetric monoidal functors. 

\begin{definition}
Let $\C$ be any $\infty$-category with pullbacks, and let $D\colon \C\op\rightarrow \text{CAlg}(\infcat^L)$ be a functor. Let $f\colon X\rightarrow Y$ be a morphism in $\C$. We will say that $f$ is \textit{$D$-locally contractible} if $f^*\colon D(Y) \to D(X)$ admits a $D(Y)$-linear left adjoint $\pfs{f}$. It is called $D$-\emph{universally locally contractible}, if for any pullback square in $\C$ of the form \eqref{pb}
\begin{enumerate}[(i)]
	\item $g$ is $D$-locally contractible, and
	\item the natural map 
	\[\pfs{g}\pb{v}\rightarrow \pb{u}\pfs{f}\]
	adjoint to 
	\[\pb{v}\rightarrow\pb{v}\pb{f}\pfs{f}\simeq\pb{g}\pb{u}\pfs{f}\]
	is invertible. 
\end{enumerate}
If $D$ is a cocomplete six functor formalism on a geometric setup $\C$, we say that a morphism $f$ in $\C$ is $D$-(universally) locally contractible if it is so with respect to the underlying pullback functoriality of $D$. In case $Y$ is a terminal object, we will also say that $X$ is $D$-(universally) locally contractible.
When $\C=\LCH$ and $D=\Sh{-}{\E}$, we say that $f$ is $\E$-locally contractible rather than $\Sh{-}{\E}$-locally contractible. If moreover $\E=\Mod_R$ for an $\bE_{\infty}$-ring $R$, we say that $f$ is $R$-locally contractible.
\end{definition}

\begin{corollary}\label[corollary]{1ULA=>loccontr}
Let $D$ be a cocomplete six functor formalism on a geometric setup $\C$. Let $f\colon X\rightarrow Y$ be a morphism in $\C$ and assume that the monoidal unit $1_X\in D(X)$ is $f$-suave. Then $f$ is $D$-universally locally contractible, we have natural equivalences
    \[\pfs{f}(-)\simeq\pfp{f}(-\otimes\omega_f) \quad \text{ and } \quad
    \pbp{f}(-)\simeq\pb{f}(-)\otimes\omega_f\]
and for any pullback diagram in $\C$ of the form \eqref{pb}, the canonical map $v^!f^* \to g^*u^!$ is invertible.
\end{corollary}
\begin{proof}
The second equivalence is simply the special case of \cref{propULA}(ii) with $F=1_X$. 
Let us consider the 2-functor 
\[\Hom{\cohcorr{D}(\C_{/Y})}{Y}{-} \colon \cohcorr{D}(\C_{/Y})\rightarrow\Cat_\infty^L\]
corepresented by $\id_Y$. 
On Hom categories, it sends the morphism $1_X$ in $\Hom{\cohcorr{\D}(\C_{/Y})}{Y}{X}$ to the functor $f^*$ and the morphism $\omega_f$ in $\Hom{\cohcorr{\D}(\C_{/Y})}{X}{Y}$ to $f_!(- \otimes \omega_f)$. Now $1_X$ being $f$-suave implies by \cref{propULA} that $\omega_f$ viewed as a morphism $X \to Y$ is left adjoint to $1_X$, viewed as a morphism $Y \to X$. The above 2-functor perserves adjunctions, so we deduce that $f_!(- \otimes \omega_f)$ is left adjoint to $f^*$. It is $D(Y)$-linear by the projection formula. That the map $\pfs{g}\pb{v}\rightarrow \pb{u}\pfs{f}$ is an equivalence for all pullback diagrams \eqref{pb} then follows from the facts that $1_W$ is $g$-suave, so that $\pfs{g} \simeq g_!(-\otimes \omega_g)$, that the canonical map $v^*\omega_f \to \omega_g$ is invertible, and base-change. Finally, by passing to left adjoints, it finally suffices to show that the map $\pfs{f}\pfp{v} \to \pfp{u} \pfs{g}$ is invertible. This follows from the same arguments.
\end{proof}

\begin{proposition}\label[proposition]{thomiso}
    Let $D$ be a six functor formalism defined on a geometric setup $\C$. Let $f\colon X\rightarrow Y$ be a morphism whose diagonal $\Delta\colon X \to X \times_Y X$ is $D$-cohomologically proper\footnote{See \cite[Definition 6.10]{scholze2022six}. Note also that ``proper'' maps are $D$-cohomologically proper, see \cite[Remark 5.4]{scholze2022six} for the dual statement.} in $\C$, and assume that the monoidal unit $1_X\in D(X)$ is $f$-suave. Then there is a natural equivalence
    $$\pfs{{p_{1}}}\pf{\Delta}\pb{f}(-)\simeq\pbp{f}(-),$$
    where $\Delta$ denotes the diagonal map $X\rightarrow X\times_Y X$.
\end{proposition}
\begin{proof}
By \cite[Proposition 6.11]{scholze2022six}, we have an equivalence $\Delta_! \simeq \Delta_*$ by the assumption that $\Delta$ is $D$-cohomologically proper. Moreover, $1_{X\times_Y X}$ is $p_1$-suave by \cref{ULAstabilityunderbasechange}. Hence \cref{1ULA=>loccontr} applies to the map $p_1\colon X \times_Y X \to X$ and to $f$. Finally we recall that the canonical map $p_2^*\omega_f \to \omega_{p_1}$ is invertible, see \cref{charunitULA}. This determines the following sequence of equivalences
\begin{align*}
	\pfs{{p_1}}\pf{\Delta}\pb{f}(-) &\simeq \pfp{{p_1}}(\pfp{\Delta}\pb{f}(-)\otimes\omega_{p_1}) \\
	&\simeq\pfp{{p_1}}\pfp{\Delta}(\pb{f}(-)\otimes\pb{\Delta}\omega_{p_1}) \\
	&\simeq\pb{f}(-)\otimes\pb{\Delta}\omega_{p_1} \\
	&\simeq\pb{f}(-)\otimes\pb{\Delta}\pb{p_2}\omega_f \\
	&\simeq\pb{f}(-)\otimes\omega_f \simeq f^!(-)
\end{align*}
as claimed.
\end{proof}

\begin{remark}\label[remark]{rmk:lowersharpdualizing=thomspTX}
Let $X$ be a topological manifold, and let $a:X\rightarrow\ast$ be the unique map. It follows for example from \cref{charlocalweakcontr} and \cref{kunneth=>1ULA=loccontr} that $\s{X}\in\Sh{X}{\spectra}$ is $a$-suave. As an immediate consequence of \cref{thomiso} and localization sequences, we see that the spectrum $\pfs{a}\omega_X$ coincides with the Thom spectrum of the tangent microbundle of $X$, as defined for example in \cite{holm1967microbundles}. 
\end{remark}

\subsection{K\"unneth}

Let $\C$ be any $\infty$-category with pullbacks and $D \colon \C\op\rightarrow \text{CAlg}(\infcat^L)$ a functor. For a cospan $X\rightarrow Z\leftarrow Y$ in $\C$, there is then a canonical functor 
\begin{equation}\label{kunncomp}
    D(X)\otimes_{D(Z)} D(Y)\rightarrow D(X\times_Z Y).
\end{equation}

\begin{definition}
We say that $D\colon \C\op\rightarrow \text{CAlg}(\infcat^L)$ satisfies \textit{the K\"unneth formula} if
for any cospan $X\rightarrow Z\leftarrow Y$ in $\C$, the functor \eqref{kunncomp} is an equivalence. 

    We say that a cocomplete six functor formalism $D$ satisfies the K\"unneth formula if its underlying pullback functoriality $\C\op\rightarrow \text{CAlg}(\infcat^L)$ satisfies the K\"unneth formula.
\end{definition}

\begin{example}
Let $\LCH$ be the category of locally compact Hausdorff spaces and $\D$ a bicomplete stable, or presentable $\infty$-category, equipped with a closed symmetric monoidal structure (e.g.\ $\D= \Sp$ or, more generally, $\D = \Mod_R$ for some $\bE_\infty$-ring spectrum $R$). Then the six functor formalism $X \mapsto \Sh{X}{\D}$ on $\LCH$, see \cite{volpe2023operations}, satisfies the K\"unneth formula. Indeed, by \cite[Corollary 1.10]{aoki2023sheavesspectrum}, the statement holds true when $\D = \Ani$\footnote{Note here that the locale functor sends pullbacks to pushouts when restricted to locally compact Hausdorff spaces (see \cite[Lemma 7.6]{volpe2023operations}).} and by \cite[Corollaries 2.24 \& 5.16]{volpe2023operations} we have $\Sh{-}{\D} \simeq \Sh{-}{\Ani}\otimes \D$. We warn the reader, that in general, the functor sending a topological space to its locale of open subsets does not preserve pullbacks. Therefore, restricting to locally compact Hausdorff spaces is indeed necessary here.
\end{example}

\begin{lemma}\label[lemma]{lemma:embedding}
Let $D$ be a cocomplete six functor formalism on a geometric setup $\C$ which satisfies the K\"unneth formula. Then the functor corepresented by a terminal object $\ast$ of $\C$ induces a 2-fully faithful embedding 
\[ \Phi\colon \cohcorr{D}(\C) \to \Mod_{D(\ast)}(\Cat_\infty^L),\]
i.e., for any two objects $X$ and $Y$ the induced functor $\Phi_{X,Y}$ on Hom categories is an equivalence.
\end{lemma}
\begin{proof}
First, observe that by \cite[Remark 4.1.7]{heyer20246}, the functor corepresented by $\ast$ induces a functor \[ \Phi\colon \cohcorr{D}(\C) \to \Mod_{D(\ast)}(\Cat_\infty^L).\] We show that $\Phi$ is fully faithful.

Let us denote by $\D(X)^\vee = \Fun_{D(\ast)}(D(X),D(\ast))$ the $D(\ast)$-linear dual of $D(X)$ and by $a$ the map $X \to \ast$. The map $D(X) \otimes_{D(\ast)} D(X) \to D(X\times X) \xrightarrow{\Delta^*} D(X) \xrightarrow{a_!} D(\ast)$ is then adjoint to a $D(\ast)$-linear map $D(X) \to D(X)^\vee$. We then consider the following commutative diagram
\[ \begin{tikzcd}
	D(X) \otimes_{D(\ast)} D(Y) \ar[r] \ar[d] & D(X)^\vee \otimes_{D(\ast)} D(Y) \ar[d] \\
	D(X\times Y) \ar[r] & \Fun_{D(\ast)}(D(X),D(Y))
\end{tikzcd}\]
in which the lower horizontal map is induced by the functor corepresented by $\ast$ on Hom categories. Now if $D$ satisfies the K\"unneth formula, the left vertical map is an isomorphism. Moreover, since the object $X$ in $\Corr(\C)$ is self-dual, witnessed  by the span $X \times X \leftarrow X \to \ast $, we find that the above constructed map $D(X) \to D(X)^\vee$ is an isomorphism and that $D(X)$ is $D(\ast)$-dualisable. This implies that in the above diagram also the top horizontal and the right vertical map is an isomorphism, and hence so is the bottom horizontal one.
\end{proof}

\begin{corollary}\label[corollary]{kunneth=>1ULA=loccontr}
Let $D$ be a cocomplete six functor formalism on a geometric setup $\C$ which satisfies the K\"unneth formula and $f\colon X \to Y$ a morphism in $\C$.
Then the following assertions are equivalent:
\begin{enumerate}[(i)]
	\item\label{item:i} $f$ is $D$-locally contractible
	\item\label{item:ii} the monoidal unit ${1}_X\in D(X)$ is $f$-suave;
	\item\label{item:iii} $f$ is $D$-universally locally contractible.
\end{enumerate}
\end{corollary}
\begin{proof}
We note that the functor $\Phi_{Y,X} \colon D(X) = \Hom{\cohcorr{\D}(\C_{/Y})}{Y}{X} \to \Fun_{D(Y)}(D(Y),D(X))$ sends $1_X$ to $f^*$. \cref{lemma:embedding} applied to $\C_{/Y}$ and the induced six functor formalism implies that the condition that $f^*$ \emph{admits} a $D(Y)$-linear left adjoint is equivalent to the condition that $1_X$, viewed as a morphism $X \to Y$ in $\cohcorr{D}(\C_{/Y})$ \emph{is} a left adjoint, which is the definition of $1_X$ being $f$-suave. Hence (i) implies (ii). By \cref{1ULA=>loccontr}, (ii) implies (iii) even without assuming that $D$ satisfies the K\"unneth formula. Finally, (iii) implies (i) by definition.
\end{proof}

\begin{corollary}\label[corollary]{sharp=shriektensdualwithkunnet}
Let $D$ be a cocomplete six functor formalism on a geometric setup $\C$ which satisfies the K\"unneth formula. Let $f\colon X\rightarrow Y$ be a map in $\C$, and assume that $\pb{f}$ admits a $D(Y)$-linear left adjoint $\pfs{f}$. Then there are natural equivalences of $D(Y)$-linear functors
\[\pfs{f}\simeq\pfp{f}(-\otimes\omega_f), \quad  \pb{f}\simeq\sHom{D(X)}{\omega_f}{\pbp{f}(-)}, \quad \text{ and } \quad  \pfs{f}\pb{f} \simeq \pfp{f}\pbp{f}.\]
\end{corollary}
\begin{proof}
The first and last equivalence follows from \cref{kunneth=>1ULA=loccontr} and \cref{1ULA=>loccontr} and the second follows from the first by passing to right adjoints. 
\end{proof}

The implication $\eqref{item:i} \Rightarrow \eqref{item:iii}$ in \cref{kunneth=>1ULA=loccontr} in fact holds true more generally. It will be convenient to record the following here.

\begin{proposition}\label[proposition]{kunnetheasyloccontr}
Let $\C$ be any $\infty$-category with pullbacks, and let $D \colon \C\op\rightarrow \CAlg(\infcat^L)$ be
a functor which satisfies the K\"unneth formula. Let $f \colon X\rightarrow Y$ be a morphism in $\C$, and assume that $\pb{f}$ admits a $D(Y)$-linear left adjoint $\pfs{f}$. Then $f$ is $D$-locally contractible. 
\end{proposition}
\begin{proof}
Consider the pullback square in $\C$ as in \eqref{pb}. The functor 
\[D(Z)\otimes_{D(Y)}(-) \colon \text{Mod}_{D(Y)}(\Cat_\infty^L) \rightarrow\text{Mod}_{D(Z)}(\Cat_\infty^L)\]
preserves adjunctions, and since $D(-)$ sends (\ref{pb}) to a pushout square of symmetric monoidal $\infty$-categories, we obtain an identification $D(Z)\otimes_{D(Y)}\pb{f}\simeq\pb{g}$. Hence, we see that the functor $\pfs{g}\coloneqq D(Z)\otimes_{D(Y)}\pfs{f}$ is a $D(Z)$-linear left adjoint of $\pb{g}$. Moreover, we have equivalences
\begin{align*}
	\pfs{g}\pb{v}&\simeq(D(Z)\otimes_{D(Y)}\pfs{f})(\pb{u}\otimes_{D(Y)}D(Z)) \\
	&\simeq \pb{u}\otimes_{D(Y)}\pfs{f} \\
	&\simeq \pb{u}\pfs{f},
\end{align*}
proving the proposition.
\end{proof}

\section{Sheaves on (LCH) spaces}

The scope of this section is to explore local contractibility and cohomological smoothness in topology. We give multiple examples of $\Ani$-locally contractible spaces, and compare the sheaf theoretical local contractibility with a more familiar notion of local weak contractibility. We thus study general conditions implying that sheaf cohomology and singular cohomology of a topological space agree. Moving on to cohomological smoothness, we highlight how this notion gives bounds to cohomological dimension. We conclude the section by showing that the shape of any locally contractible CH space, whose dualizing sheaf is locally constant, is a Poincar\'e duality complex, with Spivak normal fibration given by the tensor inverse to the dualizing sheaf.

\subsection{Local contractibility}

The goal of this section is to explore the notion of $D$-local contractibility in the special case of $D = \Sh{-}{\D}$, which we recall satisfies the K\"unneth theorem. We start by providing a convenient class of examples of spaces $X$ which are $\Ani$-locally contractible.

\begin{terminology}
Let $X$ be a topological space. We denote by $\shape(X)$ its shape in the sense of $\infty$-topos theory, i.e.\ the shape of $\Sh{X}{\Ani}$. That is, it is the composite $a_*a^* \in \Fun^{\acc,\lex}(\Ani,\Ani) \simeq \Pro(\Ani)$, where $a \colon X \to \ast$ denotes the unique map. A space is called of constant shape if its shape lies in the image of the full inclusion $\Ani \subseteq \Pro(\Ani)$, and locally of constant shape if every open subset is of constant shape. By \cite[Prop.\ A.1.8]{lurie2017higher}, a space is locally of constant shape if and only if the functor $a^* \colon \Ani \to \Sh{X}{\Ani}$ admits a left adjoint $\pfs{a}$, in which case there is a canonical equivalence $\shape(X) \simeq \pfs{a}(\ast)$. More generally, if $\C$ is a closed symmetric monoidal category with unit $1$, and $X$ is $\C$-locally contractible, we write $\sh_\C(X)$ to denote $\pfs{a}^{\C}(1)$.
\end{terminology}

\begin{lemma}
A topological space $X$ is $\Ani$-locally contractible if $X$ is locally of constant shape. If $X$ is LCH then it is $\Ani$-universally locally contractible if and only if it is $\Ani$-locally contractible.
\end{lemma}
\begin{proof}
The first part was explained above, and the second follows in addition from \cref{kunnetheasyloccontr} applied to the functor $X \mapsto \Sh{X}{\Ani}$.
\end{proof}

We collect the following relevant properties.

\begin{lemma}\label[lemma]{lemma:products}
Let $X$ be a $\Ani$-universally locally contractible space, and let $Y$ be $\Ani$-locally contractible. Then $X\times Y$ is $\Ani$-locally contractible.
\end{lemma}
\begin{proof}
	The lemma is prove by factoring the map $X\times Y\to\ast$ as $X\times Y\xrightarrow{p_2}Y\xrightarrow{a}\ast$, and observing that $\pb{p_2}$ and $\pb{a}$ admit a left adjoint respectively because $X$ is $\Ani$-universally locally contractible and $Y$ is $\Ani$-locally contractible.
\end{proof}

\begin{lemma}\label[lemma]{retractloccontr}
Let $Y$ be $\Ani$-(universally) locally contractible space. Suppose that $X$ is a retract of an open subset of $Y$. Then $X$ is $\Ani$-(universally) locally contractible.
\end{lemma}
\begin{proof}
Since open subsets of $\Ani$-locally contractible spaces are $\Ani$-locally contractible, we may assume that $X$ is a retract of $Y$. Let $i \colon X \to Y$ and $r \colon Y \to X$ be continuous maps with $ri = \id_X$ and let $b \colon Y \to \ast$ and let $a=bi$. Then we the composite
\[\pb{a}\xrightarrow{\text{unit}}\pf{r}\pb{r}\pb{a}\simeq\pf{r}\pb{b}\xrightarrow{\text{unit}}\pf{r}\pf{i}\pb{i}\pb{b}\simeq\pb{a}\]
 is the identity, so $a^*$ is a retract of the functor $r_*b^*$ which has a left adjoint $\pfs{b}r^*$, and consequently has a left adjoint itself since $\Ani$ is idempotent complete (see \cite[Lemma 21.1.2.14]{lurie2018spectral}). For $\Ani$-universal local contractibility, we observe that, if $Z$ is any topological space, then $Z\times-$ preserves open immersions and retractions.
\end{proof}

\begin{lemma}\label[lemma]{lemma:locally-bla}
Let $\{U_i\}_{i \in I}$ be an open cover of a topological space $X$. Suppose that for every $i \in I$, $U_i$ is $\Ani$-locally contractible. Then $X$ is $\Ani$-locally contractible.
\end{lemma}
\begin{proof}
See \cite[Corollary A.1.7]{lurie2017higher}.
\end{proof}

\begin{lemma}\label[lemma]{lemma:homotopy-shape-equivalence}
Let $X \to Y$ be a homotopy equivalence. Then it induces an isomorphism upon applying $\shape(-)$.
\end{lemma}
\begin{proof}
See \cite[Remark A.4.6]{lurie2017higher}.
\end{proof}

\begin{lemma}\label[lemma]{infiniteprodloccontr}
Let $\{X_i\}_{i\in S}$ be a family of topological spaces indexed by a set $S$ and let $X=\prod_{i\in S}X_i$ be their product. Suppose that the family has the following properties.
\begin{enumerate}[(i)]
    \item For each $i\in S$, $X_i$ is $\Ani$-universally locally contractible.
    \item All but finitely many $X_i$ are contractible.
\end{enumerate}
Then $X$ is $\Ani$-locally contractible.
\end{lemma}
\begin{proof}
Recall that the product topology on $X$ has a basis $\mathcal{B}$ which consists of sets of the form $\prod_{i\in S}U_i$, where each $U_i$ is open in $X_i$ 
and there exists a finite subset $F\subseteq S$ such that for all $j\in S\setminus F$, we have $U_j = X_j$. Since this basis is closed under finite intersections, using \cref{lemma:locally-bla}, it then suffices to show that every element in the basis has constant shape. Without loss of generality, an element of the basis is given by $\prod_{f \in F} U_f \times \prod_{i \in S \setminus F} X_i$, for $F \subseteq S$ a finite set, with $\prod_{i \in S\setminus F} X_i$ contractible. Hence the map $\sh(\prod_{f \in F} U_f \times \prod_{i \in S \setminus F} X_i) \to \sh(\prod_{f \in F} U_f)$ is an equivalence, see \cref{lemma:homotopy-shape-equivalence}. Hence, we conclude by \cref{lemma:products}.
\end{proof}

\subsection{Hypercompleteness}
In this subsection, we record a useful consequence of $\C$-local contractibility. To that end, we first recall some basic properties of $\Ani$- or $\Mod_R$-valued sheaves, where $R$ is a connective $\bE_\infty$-ring. Let $\C$ be $\Ani$ or $\Mod_R$. Recall that a sheaf is hypercomplete if it is local with respect to $\infty$-connective maps and that a map of sheaves is $\infty$-connective if and only if it induces an isomorphism on stalks. Pullback functors preserve $\infty$-connected maps.

In particular, the stalk functors are jointly conservative if and only if $\Sh{X}{\C}$ is hypercomplete. Moreover, if $X$ is paracompact with finite covering dimension, then $\Sh{X}{\C}$ is hypercomplete \cite[Cor.\ 7.2.1.12 \& Theorem 7.2.3.6]{lurie2009higher}. 
The standard $t$-structure on $\Mod_R$ induces one on $\Sh{X}{\Mod_R}$. In this $t$-structure, for $f\colon X \to Y$, the pullback functor $\pb{f} \colon \Sh{Y}{\Mod_R} \to \Sh{X}{\Mod_R}$ is $t$-exact. Moreover, a morphism in $\Sh{X}{\Mod_R}$ is $\infty$-connective, if and only if its fibre is $\infty$-connective, if and only if the fibre has trivial stalks. We also note that the extension and restriction of scalar adjunction associated to a map of connective $\bE_\infty$-rings $R \to S$ induces an adjunction on the level of sheaves for which both functors are compatible with the pullback functoriality of sheaves. In particular, an $\infty$-connective sheaf of $S$-modules is also $\infty$-connective as sheaf of $R$-modules, and therefore $R$-hypercomplete spaces are also $S$-hypercomplete.

\begin{lemma}\label[lemma]{loccontr=>locconstishyp}
Let $\D$ be either $\Ani$ or $\text{Mod}_R$ for some connective $\bE_\infty$-ring spectrum $R$. Let $X$ be topological space which is $\D$-locally contractible. Then constant $\D$-valued sheaves are hypercomplete.
\end{lemma}
\begin{proof}
Let $F\to G$ be an $\infty$-connective morphism in $\Sh{X}{\D}$. To see that constant sheaves are hypercomplete, we need to show that $\Hom{}{G}{a^*(M)} \to \Hom{}{F}{a^*(M)}$ is an equivalence for any $M\in\C$. Since $X$ is $\D$-locally contractible, this is equivalent to showing that $a_\sharp(F) \to a_\sharp(G)$ is an equivalence in $\D$. This follows from the observation that, as a left adjoint to a left exact ($t$-exact when $\D=\text{Mod}_R$) functor, $a_\sharp$ preserves connective objects and that $\infty$-connective objects in $\D$ are trivial.
\end{proof}

\subsection{Comparison between sheaf and singular cohomology}
In this next section, we provide many examples of  $\Ani$-locally contractible spaces. Our approach is based on studying the relation between singular cohomology and sheaf cohomology, by relating the shape to the singular complex. The results in this subsection were also obtained by Philippe Vollmuth in his Master's thesis \cite{Vollmuth}, variants of them have appeared in \cite{Sella}, \cite{petersen2022remark}, \cite{clausen-lectures}, and \cite{haine2023homotopy}. With paracompactness assumption, they can be found in \cite{bredon2012sheaf}. We discuss the precise relations to these previous works in \cref{rmk:comparison}.

Let $X$ be any topological space. Write $\Sing\colon\Op{X}\rightarrow\Ani$ for the functor which associates to each open $U\subseteq X$ its weak homotopy type $\sing{U}\in\Ani$. By \cite[Proposition A.3.2]{lurie2017higher} and \cite[Lemma A.3.10]{lurie2017higher}, $\Sing$ is a hypercomplete cosheaf. 

\begin{construction}
Let $\D$ be a complete $\infty$-category so that $\D$ is cotensored over $\Ani$. 
For each $M\in\C$, the association $U\mapsto M^{\Sing(U)}$
defines a $\D$-valued presheaf on $X$ which we denote by $\pb{a}_{\Sing_{\D}}(M)$. Moreover, since $\Sing$ is a cosheaf, this construction provides a functor
\[\pb{a}_{\Sing_{\D}}\colon \D\rightarrow\Sh{X}{\D}\]
which is readily checked to preserve limits.
If $\D$ is a bicomplete $\infty$-category, then it is in particular tensored over $\Ani$, and so one may define the functor 
\begin{equation}\label{coshsingD}
	\Op{X} \times \D \to \D,\quad (U,M) \mapsto \Sing(U)\otimes M.
\end{equation} Notice that \eqref{coshsingD} is a cosheaf in the first argument and preserves colimits in the second. Therefore, if $\D$ and $X$ are such that the canonical functor 
\begin{equation}\label{tensshaniC}
	\Sh{X}{\Ani}\otimes \D\rightarrow\Sh{X}{\D}
\end{equation}
is an equivalence\footnote{This is the case for example if $\D$ is presentable or $X\in\LCH$ and $\D$ is stable.}, \eqref{coshsingD} extends uniquely to a colimit preserving functor
\[\pfs{a}^{\Sing_{\D}}\colon \Sh{X}{\D}\rightarrow\D.\]
By construction, $\pfs{a}^{\Sing_{\D}}$ is left adjoint to $\pb{a}_{\Sing_{\D}}$. If $\D$ is equipped with a closed symmetric monoidal structure with unit $1$, we will often write $\Sing_{\D}(X)$ to denote $\pfs{a}^{\Sing_{\D}}(1)$.

When the inclusion $\Sh{X}{\D}\hookrightarrow\Fun(\Op{X}\op,\D)$ admits a left adjoint (e.g.\ if $\D$ is presentable or $X\in\text{LCH}$ and $\D$ is stable and bicomplete) called sheafification, the global section functor $\pf{a}^{\D}$ admits a left adjoint $\pb{a}_{\D}$ given by sheafifying the constant presheaf. Notice that the unique map $\Sing(X)\rightarrow\ast$ in $\Ani$ induces a natural transformation $\id_{\D}\rightarrow\pf{a}\pb{a}_{\Sing_{\D}}$. Since $\pb{a}_{\Sing_{\D}}$ is a sheaf, by adjunction we get a natural comparison transformation
\begin{equation}\label{sheafvssingcoh}\tag{$\comp$}
    \comp\colon \pb{a}_{\D}\rightarrow\pb{a}_{\Sing_{\D}}.
\end{equation}
\end{construction}

\begin{remark}\label[remark]{rmk:factcompatstalk}
	Let $x$ be any point in $X$. Let $j:U\hookrightarrow X$ be any open containing $x$, and let $b:U\rightarrow\ast$ be the unique map. By construction, we have $\pb{j}_{\D}\pb{a}_{\Sing_{\D}}\simeq\pb{b}_{\Sing_{\D}}$. Since $\pb{x}_{\D}$ is left inverse to $\pb{b}_{\D}$, we deduce that for any $M\in\D$, the map $\pb{x}_{\D}(\comp_M)$ in $\D$ is homotopic to the composition 
	$$M\rightarrow M^{\Sing(U)}\rightarrow\varinjlim\limits_{x\in V}M^{\Sing(V)},$$
	where the first map is induced by $\Sing(U)\rightarrow\ast$.
\end{remark}

We wish to understand in which situations the map (\ref{sheafvssingcoh}) is invertible. We begin with the following.

\begin{lemma}\label[lemma]{sheafvssingcohwithboundedcoeff}
Let $R$ be a connective $\bE_\infty$-ring and let $M\in\Mod_R$ be bounded above. Assume that the map $\comp_M\colon \pb{a}_{R}(M)\rightarrow\pb{a}_{\Sing_{R}}(M)$ induces an isomorphism on stalks. Then $\comp_M$ is an isomorphism.
\end{lemma}
\begin{proof}
Since $\pb{a}_{\Sing_{R}}(M)$ is hypercomplete, it suffices to show that $\pb{a}_{R}(M)$ is also hypercomplete. Since $\pb{a}_{R}$ is $t$-exact, we see that $\pb{a}_{R}(M)$ is bounded above, and therefore hypercomplete.
\end{proof}

The following is an immediate, yet important, consequence of invertibility of \eqref{sheafvssingcoh}.

\begin{corollary}\label[corollary]{sheafvssingeqCniceimpliesShcloccontr}
Let $X$ be a topological space, $\D$ be a bicomplete $\infty$-category such that the canonical functor \eqref{tensshaniC} is an equivalence. If the natural transformation \eqref{sheafvssingcoh} is invertible, then $X$ is $\D$-locally contractible. 
\end{corollary}
\begin{proof}
Indeed, we find that $\pfs{a}^{\Sing_{\D}}$ is a left adjoint of $\pb{a}_\D$.
\end{proof}

We therefore wish to establish point-set topological conditions on $X$ that imply that (\ref{sheafvssingcoh}) is invertible. We axiomatize such a condition in the following definition.

\begin{definition}\label[definition]{singcontractibility}
Let $X$ be any topological space, and let $\D$ be a complete $\infty$-category. We say that $X$ is \emph{$\D$-locally weakly contractible}, if for any $M\in\D$, $x \in X$ and open neighbourhood $x\in U\subseteq X$, there is an open $V \subseteq U$ containing $x$ and a commuting triangle
\[\begin{tikzcd}
	& M^{\sing{\{x\}}} \arrow[rd] & \\
	M^{\sing{U}} \arrow[ru] \arrow[rr] && M^{\sing{V}}
\end{tikzcd}\]
in $\D$. Notice that if $X$ is $\D$-locally weakly contractible, then so is any open subset $U$ of $X$.
\end{definition}

For $\D=\Mod_R$, we will again say $R$-locally weakly contractible rather than $\Mod_R$-locally weakly contractible. \cref{singcontractibility} is intended to ensure that sheaf cohomology agrees with singular cohomology, and likewise that the shape agrees with the weak homotopy type when $\D=\Ani$. However, as we will see, this condition alone does not suffice, as one must additionally impose a mild hypercompleteness assumption.

\begin{remark}\label[remark]{charlocalweakcontr}
Suppose that $\D$ is a bicomplete $\infty$-category equipped with a closed symmetric monoidal structure. Denote by $1$ the monoidal unit in $\D$, by $\otimes$ its tensor product and by $\sHom{\D}{-}{-}$ the internal hom. Then, for any $A\in\Ani$ and $M\in\D$, there is an isomorphism $$M^A\simeq\sHom{\D}{A\otimes 1}{M}$$ in $\D$ which is natural in $M$ and $A$. Therefore, we deduce that a sufficient condition for a topological space $X$ to be $\D$-locally weakly contractible is requiring that, for any $x \in X$ and open neighbourhood $x\in U\subseteq X$, there is an open $V \subseteq U$ containing $x$ and a commuting triangle
\begin{equation}\label{condlocalcontrtens1}
\begin{tikzcd}
	& \Sing_{\D}(\{x\}) \arrow[rd] & \\
	\Sing_{\D}(V) \arrow[ru] \arrow[rr] && \Sing_{\D}(U)
\end{tikzcd}
\end{equation}
in $\D$.
\end{remark}

\begin{remark}
In classical topology references, the condition considered in (\ref{condlocalcontrtens1}) is often referred to as (homological) local weak contractiblity. 
\end{remark}

\begin{example}\label[example]{anrlocweakcontr}
We recall that an ANR (absolute neighborhood retract) is a separable, metrizable space $X$ such that whenever it is embedded as a closed subspace of another metrizable space $Z$, there is an open subset $U$ of $Z$ containing $X$ as a retract. By \cite[Theorem 2]{fox1942characterization} and \cref{charlocalweakcontr}, locally compact ANRs are $\Ani$-locally weakly contractible.
In general, a locally compact ANR is not locally contractible in the stronger sense that every point has a neigbourhood basis consisting of weakly contractible open subsets. Topological manifolds, and more generally ENRs, are ANRs satisfying this stronger local contractibility property.
\end{example}

We now move on to investigate whether the map (\ref{sheafvssingcoh}) is an isomorphism on stalks for $X$ $\D$-locally weakly contractible. We will make use of the following preliminary results. To that end, Let $P$ be a poset and write $P^{[1]}$ for the poset of arrows in $P$, $\mathrm{const} \colon P \to P^{[1]}$ for the functor sending an object to its identity, and $P^{[1]}\setminus P$ for the full subposet of $P^{[1]}$ given by the non-identity arrows. 

\begin{lemma}\label[lemma]{lemma:posetarrowcofinal}
	 The following statements hold true.
	\begin{enumerate}
		\item\label{item:lemmaposet1} The functor $\mathrm{const}\colon P\rightarrow P^{[1]}$ sending an object to the corresponding identity arrow is (co)final.
		\item\label{item:lemmaposet2} Assume that $P$ is cofiltered and has no initial object. Then $P^{[1]}\setminus P$ is cofiltered.
		\item\label{item:lemmaposet3} Assume that $P$ is cofiltered and has no initial object. Then the inclusion $P^{[1]}\setminus P\hookrightarrow P^{[1]}$ is cofinal.\footnote{Here, we follow the terminology of \cite{cisinski}, that is, cofinal functors induce equivalences on limits, while final functors induce equivalences on colimits.}
	\end{enumerate}
\end{lemma}

\begin{proof}
(1) follows from the fact that $\mathrm{const}$ admits left and right adjoints given by evaluating at $1,0\in[1]$, respectively. 
	
	Assume now that $P$ is cofiltered and has no initial object. To show (2) pick $f,g\in P^{[1]}\setminus P$. Since $P$ is cofiltered, we may find $h(1)\in P$ with the property that $h(1)\leq f(1), g(1)$. The assumption that $P$ has no initial object implies that one may find $h(0)\leq h(1), f(0), g(0)$ such that $h(0)\neq h(1)$. Therefore one has an arrow $h\in P^{[1]}\setminus P$ with $h\leq f, g$, as desired.
	
	To show (3), by Quillen's Theorem A, we need to show that, for any arrow $f\in P^{[1]}$, the slice $(P^{[1]}\setminus P)_{/f}$ is weakly contractible. If $f$ is not an identity arrow, $(P^{[1]}\setminus P)_{/f}$ has a terminal object, and is therefore weakly contractible. If $f=\id_p$ for some $p\in P$, one checks that the slice $(P^{[1]}\setminus P)_{/f}$ identifies with the full subposet of $P^{[1]}\setminus P$ spanned by the arrows $g$ such that $g(1)\leq p$. This poset is cofiltered, as is shown exactly as in the proof of (2), and therefore it is weakly contractible.
\end{proof}

\begin{remark}
If one drops the assumption that $P$ has no initial object, statements (\ref{item:lemmaposet2}) and (\ref{item:lemmaposet3}) in \cref{lemma:posetarrowcofinal} are no longer true. For example, if  $P$ is the horn $\Lambda_0^2$, $P^{[1]}\setminus P$ is the discrete poset consisting of two elements given by the arrows $0\leq 1$, $0\leq 2$, which is neither cofiltered nor final in $P^{[1]}$, as the latter has an initial object given by $\id_0$. 
\end{remark}

\begin{corollary}\label[corollary]{cor:colimourofarrowconst}
	Let $P$ be a cofiltered poset having no initial object. Let $\D$ be an $\infty$-category and $F\colon (P^{[1]})\op\rightarrow\D$ any functor. Assume that the restriction of $F$ to $(P^{[1]}\setminus P)\op$ is constant. Then for any $f\in P^{[1]}\setminus P$ the map
	\[F(f)\xrightarrow{F(\id_{f(0)}\leq f)}F(\id_{f(0)})\rightarrow\colim\limits_{P\op}F_{|P\op}\]
	is invertible.
\end{corollary}

\begin{proof}
	Consider the diagram 
	\[\begin{tikzcd}
		F(f) \ar[r, "{F(\id_{f(0)}\leq f)}"] \ar[d] & F(\id_{f(0)}) \ar[r] & \colim\limits_{P\op} F_{|{P\op}} \ar[d] \\
		\colim\limits_{(P^{[1]}\setminus P)\op} F_{|{(P^{[1]}\setminus P)\op}} \ar[rr] && \colim\limits_{(P^{[1]})\op} F
	\end{tikzcd}\]
	in $\D$ which commutes by definition of a colimit. By \cref{lemma:posetarrowcofinal}~\eqref{item:lemmaposet2} and the assumption that $F_{|_{(P^{[1]}\setminus P)\op}}$ is constant, the left vertical map is invertible, and by \cref{lemma:posetarrowcofinal}~\eqref{item:lemmaposet1} and~\eqref{item:lemmaposet3}, the right vertical map and the lower horizontal map are invertible, respectively. Therefore, the upper horizontal map is also invertible, as desired.
\end{proof}

We are now ready to prove our main result of the section. We first treat the following special case.

\begin{lemma}\label[lemma]{lemma:initialopencompiso}
	Let $\D$ be a complete $\infty$-category, and let $X$ be a $\D$-locally weakly contractible space. Let $x\in X$ be a point, and let $\Op{X,x}$ be the poset of all open neighbourhoods of $x$. Assume that $\Op{X,x}$ has an initial object $U$. Then the map \eqref{sheafvssingcoh} induces an equivalence on the stalk at $x$.
\end{lemma}

\begin{proof}
	Let $M$ be any object in $\D$. By \cref{rmk:factcompatstalk}, it suffices to show that the map $M\rightarrow M^{\Sing(U)}$ induced by $\Sing(U)\rightarrow\ast$ is an equivalence. We have a diagram in $\D$ 
\[\begin{tikzcd}
		& M^{\Sing(\{x\})} \ar[d] \ar[rd, "\id"] &  \\
		M^{\Sing(U)} \ar[ru] \ar[r, "\id"'] & M^{\Sing(U)} \ar[r]  & M^{\Sing(\{x\})},
\end{tikzcd}\]
	where the left triangle commutes since $X$ is $\D$-locally weakly contractible and $U$ is initial, and right triangle commutes since $\{x\}\hookrightarrow U$ is a section of $U\rightarrow\{x\}$. Hence we deduce that $\pb{x}(\comp)$ is invertible as desired.
\end{proof}

\begin{proposition}\label[proposition]{Clocwcontr=>sheafvssinginftyconn}
Let $\D$ be a complete $\infty$-category, and 
assume there is a conservative functor $\Phi \colon \D\rightarrow\Ab$ which preserves filtered colimits.\footnote{For instance, if $\D$ is a dualizable presentable stable $\infty$-category, e.g.\ itself compactly generated: In the latter case, $\Phi$ can be chosen to be the sum over all compact objects $d$ and all integers $n$ of the functors $\pi_n\hom_\D(d,-)$ and any dualizable presentable stable $\infty$-category admits a fully faithful left adjoint to a compactly generated one.} Let $X$ be a $\D$-locally weakly contractible topological space. Then the map \eqref{sheafvssingcoh} induces an equivalence on stalks.
\end{proposition}

\begin{proof}
    Since $\Phi$ is conservative and preserves filtered colimits, using again \cref{rmk:factcompatstalk}, it suffices to check that, for any $x\in X$ and $M\in\D$, the map in $\E$
\[
    \Phi(M)\rightarrow\Phi(\colim\limits_{U\in\Op{X,x}}M^{\Sing(U)})\simeq\colim\limits_{U\in\Op{X,x}}\Phi(M^{\Sing(U)})
\]
   induced by $\{x\}\to \Sing(U)$ is invertible. 
    

    Denote by $\mathcal{N}^M(X,x)$ the subcategory of $\Op{X,x}$ whose objects are the open neighbourhoods of $x$ and whose morphisms are open inclusions $U\subseteq V$ which are either the identity, or, if the inclusion is strict, such that there is a commuting triangle
\[\begin{tikzcd}
                                     & M^{\sing{\{x\}}} \arrow[rd] &                \\
M^{\sing{V}} \arrow[ru] \arrow[rr] &                           & M^{\sing{U}}
\end{tikzcd}\]
    in $\D$ and denote by $i\colon \mathcal{N}^M(X,x) \to \Op{X,x}$ the inclusion functor. Since $X$ is $\D$-locally weakly contractible, $\mathcal{N}^M(X,x)$ is cofiltered. Since the slices of $i$ over $U$ are given by $\mathcal{N}^M(U,x)$, we deduce that $i$ is cofinal. By \cref{lemma:initialopencompiso}, it suffices to treat the case where $\mathcal{N}^M(X,x)$ has no initial object. The assumption that $X$ is $\D$-locally weakly contractible implies that the restriction of the functor $U \mapsto \Phi(M^{\Sing(U)})$ to $\mathcal{N}^M(X,x)\op$ extends to a functor $(\mathcal{N}^M(X,x)^{[1]})\op\rightarrow\Ab$ as follows. On objects, send $\id_U$ to $\Phi(M^{\Sing(U)})$ and strict inclusions $U \subseteq V$ in $\mathcal{N}^M(X,x)$ to $\Phi(M)$. On morphisms where either the source or the target  are not identities, one has to describe maps of the kind $\Phi(M^{\Sing(U)}) \to \Phi(M)$, $\Phi(M) \to \Phi(M^{\Sing(U)})$, or $\Phi(M) \to \Phi(M)$ and one chooses the ones induced by $\{x\} \to \Sing(U) \to \{x\}$, and the identity of $\ast$. From the assumption that $X$ is $\D$-locally weakly contractible, we find that this construction is indeed a functor, which is constant on the full subcategory of non-identity morphisms. Therefore the proof is concluded by appealing to \cref{cor:colimourofarrowconst}.
\end{proof}


\begin{remark}
The most important case of an $\infty$-category $\D$ which does not satisfy the assumptions of \cref{Clocwcontr=>sheafvssinginftyconn} is $\Ani$. Indeed, if there were such a conservative functor $\Ani \to \Ab$ preserving filtered (even sequential) colimits, it would follow from \cite[Lemma 12.1]{Bachmann2017NormsIM}, that the canonical map $\coprod_n B\Sigma_n \to \Z \times B\Sigma_{\infty}$ is a group completion, which it famously isn't.
As a result, we do not know if the map \eqref{sheafvssingcoh} is an equivalence on stalks for all $\Ani$-locally weakly contractible topological spaces. However, under some further hypothesis on the local topology of such spaces, we can prove this:
%
%
%
\end{remark}

\begin{proposition}\label[proposition]{prop:unstableCLocwcontr=>compiso}
	Let $\D$ be a complete $\infty$-category, and let $X$ be a $\D$-locally weakly contractible topological space which is first countable. Then the map \eqref{sheafvssingcoh} induces an equivalence on stalks.
\end{proposition}
\begin{proof}
	Let $x$ be a point in $X$, and let $M$ be any object in $\D$. Pick a countable neighbourhood basis $\{W_1,W_2,\dots\}$ at $x$ and define $V_i = \cap_{j\leq i} W_j$ which is now a nested sequence of open neighbourhoods of $x$. By the $\D$-local contractibility assumption on $X$, we can then find a subsequence $\{U_n\}_{n\in \mathbb{N}}$ such that each map $U_{n+1} \to U_n$ has the property that there exists a commuting triangle in $\D$
\[ \begin{tikzcd}
	& M^{\Sing(\{x\})} \ar[dr] & \\
	M^{\Sing(U_n)} \ar[ur] \ar[rr] && M^{\Sing(U_{n+1})}
\end{tikzcd}\]

Now, the sequence $\{U_n\}_{n \in \mathbb{N}}$ determines a functor $(\mathbb{N},\geq) \to \Op{X,x}$ which is final, as its slices over $U$ are given again by $(\mathbb{N},\geq)$ which is a contractible poset, and composing the triangles as above when $n$ increases, we find that 
\[ \colim\limits_{U \in \Op{X,x}} M^{\Sing(U)} \simeq \colim\limits_{n\geq 0} M^{\Sing(U_n)} \simeq \colim\limits_{n\geq 0} M^{\Sing(\{x\})} \simeq M\]
where the transition maps at stage $n$ in the latter colimit are induced by the composites $\{x\} \to \Sing(U_n) \to \{x\}$, so that the colimit canonically identifies with $M$ as needed.
\end{proof}

\begin{corollary}\label[corollary]{equivsheafcohhom}
Let $R$ be a connective $\bE_{\infty}$-ring spectrum and let $X$ be an $R$-locally weakly contractible space. Then $X$ is $R$-locally contractible if and only if all $\Mod_R$-valued constant sheaves on $X$ are hypercomplete. In this case, we have isomorphisms
\begin{align*}
	\sing{X}\otimes M\simeq\pfs{a}^{R}\pb{a}_{R}M  \quad \text{ and } \quad  M^{\sing{X}} \simeq\pf{a}^{R}\pb{a}_{R}M
\end{align*}
for any $M\in\Mod_R$. When $X$ is LCH, we furthermore have
\[\sing{X}\otimes M\simeq\pfp{a}^{\D}\pbp{a}_{\D}M.\]
\end{corollary}

\begin{proof}
If $X$ is $\D$-locally contractible, then constant sheaves are hypercomplete by \cref{loccontr=>locconstishyp}. The converse follows from \cref{Clocwcontr=>sheafvssinginftyconn} and \cref{sheafvssingeqCniceimpliesShcloccontr}, using that $X$ is $\D$-weakly locally contractible and that $\pb{a}_{\Sing_{\D}}$ is also hypercomplete. 
The first two isomorphisms then follow immediately from \cref{Clocwcontr=>sheafvssinginftyconn} and \cref{sheafvssingeqCniceimpliesShcloccontr} and the third from \cref{sharp=shriektensdualwithkunnet}.
\end{proof}

We note that when $M=R$, the first equivalence specializes to the equivalence $\Sing_R(X) \simeq \sh_R(X)$ relating the singular complex of $X$ with its $R$-shape, and that the second equivalence follows from the first. Passing to homotopy groups, the second equivalence gives an isomorphism between sheaf cohomology and singular cohomology with coefficients $M$ for some $R$-module $M$, i.e.\ we obtain $H^*_{\mathrm{sheaf}}(X;M) \cong H^*_{\mathrm{sing}}(X;M)$. 

\begin{corollary}\label[corollary]{equivshandsing}
	Let $X$ be a locally metrizable $\Ani$-locally weakly contractible space. Then $X$ is $\Ani$-locally contractible if and only if all $\Ani$-valued constant sheaves are hypercomplete. In this case, we have isomorphisms
	\begin{align*}
		\sing{X}\otimes M\simeq\pfs{a}^{\Ani}\pb{a}_{\Ani}M  \quad \text{ and } \quad  M^{\sing{X}} \simeq\pf{a}^{\Ani}\pb{a}_{\Ani}M
	\end{align*}
	for any $M\in\Ani$.
\end{corollary}
\begin{proof}
	The proof is identical to that of \cref{equivsheafcohhom}, the only difference being that one should cite \cref{prop:unstableCLocwcontr=>compiso} instead of \cref{Clocwcontr=>sheafvssinginftyconn}.
\end{proof}

\begin{remark}\label[remark]{rmk:counterexampleforlocwc=>shvloccontr}
In general, it is not true that a locally metrizable $\Ani$-locally weakly contractible space is $\Ani$-locally contractible, in fact not even $\Z$-locally contractible. To see this we first note that if $X$ is a $\Ani$-locally contractible compact Hausdorff space, then $\sh(X)$ is a compact anima (see \cite[Proposition A.6]{volpe2022verdier}).
In \cite{borsuk1948espace} Borsuk constructs a compact metric space $X$ which is $\Ani$-locally weakly contractible, but has infinitely many non-zero Betti numbers. In particular, this space is \emph{not} $\Z$-locally contractible,\footnote{Equivalently, constant sheaves on $X$ are in general not hypercomplete.} as else $\sh(X) \simeq \Sing(X)$ were compact, contradicting the existence of infinitely many non-zero Betti numbers.
\end{remark}

\begin{remark}\label[remark]{rmk:comparison}
We record here the relation to other previous work. 
\begin{enumerate}
\item In \cite[Corollary 3.31]{haine2023homotopy}, the authors provide a comparison isomorphism between singular and sheaf cohomology of a space $X$ under the assumption that $X$ admits a basis consisting of weakly contractible open subsets. Since such spaces are evidently $\D$-locally weakly contractible, \cref{equivsheafcohhom} gives a generalization of \cite[Corollary 3.31]{haine2023homotopy}.
\item \cref{equivsheafcohhom} also generalizes the main result of \cite{petersen2022remark} and \cite{Sella} from cohomology with coefficients in an abelian group, which was also proven by Clausen in his lectures on algebraic de Rham cohomology \cite{clausen-lectures}, to cohomology with arbitrary (constant) coefficients. In fact, Petersen's assumptions on $X$ are essentially the requirement that the map \ref{sheafvssingcoh} is an equivalence on stalks. 
\end{enumerate}
\end{remark}

\begin{corollary}\label[corollary]{LCANRloccontrlocsingshape}
Let $X$ be a locally compact ANR. Then $X$ is $\Ani$-locally contractible and there is an isomorphism $\shape(X)\simeq\Sing(X)$ in $\Ani$. 
\end{corollary}
\begin{proof}
First, observe that, for any set $S$, the space $\rnum^S$ is $\Ani$-locally contractible. Indeed, this follows by \cref{infiniteprodloccontr} and the fact that $\rnum$ is contractible and $\Ani$-locally contractible\footnote{Notice that this could be deduced from the results of this paper and \cite[Theorem 7.2.3.6]{lurie2009higher}. Indeed, $\rnum$ is paracompact and of covering dimension $1$ so that their sheaf topoi are hypercomplete, and in particular, constant sheaves are hypercomplete. Therefore, since $\rnum$ is $\Ani$-locally weakly contractible, we get that it is $\Ani$-locally contractible by \cref{equivshandsing}.}.

Now, any separable and metrizable LCH space
is completely metrizable, since it is open in its completion, and open subsets of completely metrizable spaces are completely metrizable \cite{Willard}. Every separable, completely metrizable space can be embedded as a closed subset of $\rnum^\omega$ \cite[Corollary 4.3.25]{Engelking}.\footnote{Likewise, compact ANRs can be embedded as closed subsets of $[0,1]^\omega$.} Hence, a locally compact ANR is a retract of an open subset of $\rnum^\omega$, so the result follows from \cref{retractloccontr} and \cref{equivshandsing}.
\end{proof}

In particular, locally compact ANRs are locally of singular shape in the sense of \cite{lurie2017higher}, see also  \cite{Milnor_ANRS}.

\subsection{$\D$-smoothness and cohomological dimension}
\label{section:cohdim}

Here, we study the relation between  $\D$-smoothness and an appropriate notion of cohomological dimension. We begin with the following inheritance property that we will use on occasion. This is used later to give a detailed comparison of our results with the work of Bredon and Wilder on homology manifolds. Beware that we are using homological grading.

\begin{definition}
Let $R$ be a connective $\bE_{\infty}$-ring, $X$ be a LCH space, and $n\geq 0$ a natural number. We say that $X$ has \emph{$R$-$!$-cohomological dimension at most $n$} if, for any $F\in(\Sh{X}{\text{Mod}_R})_{\geq0}$, we have that $\pfp{a}F$ belongs to $(\text{Mod}_R)_{\geq -n}$. In this case, we write $\dim^!_R(X)\leq n$. We say that $X$ has \emph{$R$-$!$-cohomological dimension equal to $n$} if $\dim^!_R(X)\leq n$ and there exists $G\in \Sh{X}{\text{Mod}_R}_{\geq0}$ such that $\pi_{-n}(\pfp{a}G)\neq 0$.
\end{definition}

We record equivalent charaterizations of the condition $\dim^!_R(X)\leq n$. In particular, we see that our definition of !-cohomological dimension coincides with the classical definition of cohomological dimension on compact metrizable spaces. In what follows, for any closed subset $A\subseteq X$, we denote by $\Sec{X,A}{R}$ the fiber of the restriction map $\Sec{X}{R}\rightarrow\Sec{A}{R}$. 

\begin{proposition}\label[proposition]{prop:charcohdim}
	Let $X$ be any LCH space, and let $R$ be any connective $\bE_{\infty}$-ring. Consider the following statement.
	\begin{enumerate}[(1)]
		\item $\dim^!_R(X)\leq n$.
		\item For any open subset $U\subseteq X$, $\cSec{U}{R}$ belongs to $(\text{Mod}_R)_{\geq -n}$.
		\item For any closed subset $A\subseteq X$, $\Sec{X,A}{R}$ belongs to $(\text{Mod}_R)_{\geq -n}$.
		\item For any closed subset $A\subseteq X$, $\pi_{-n-1}\Sec{X,A}{R}=0$.
	\end{enumerate}
    Then (1) is equivalent to (2). If $X$ is compact and metrizable, then all statements are equivalent.
\end{proposition}
\begin{proof}
	Clearly (1) implies (2). To see that (2) implies (1), one combines the following facts: $\pfp{a}(-)=\cSec{X}{-}$ preserves colimits, $(\text{Mod}_R)_{\geq -n}$ is closed under colimits in $\Mod_R$ and that any sheaf of $R$ modules can be written as a colimit of sheaves of the from $\pfp{j}R_U$, for $j:U\hookrightarrow X$ an open immersion. To see that (2) is equivalent to (3) when $X$ is a compact, observe that, using the localization sequences (see \cite[Section 4]{volpe2023operations}), we have an equivalence $\cSec{X\setminus A}{R}\simeq \Sec{X,A}{R}$. The fact that for compact metrizable spaces (3) is equivalent to (4) is classical. See for example Appendix A in \cite{walsh2006dimension}, \cite{dranishnikov1990generalized} right after Definition 2, and \cite[Proposition 7.1.0.1]{lurie2009higher} for the isomorphism between $\pi_m(\Sec{A}{R})$ and homotopy classes of maps $A\rightarrow \mathcal{R}_m$ for all $A\subseteq X$ closed and $m\leq 0$, where $\mathcal{R}_m$ denotes a CW model of the anima $\Omega^{\infty+m}R$. 
\end{proof}

\begin{remark}
	Classically, a compact metrizable space $X$ is said to have \textit{$R$-cohomological dimension $\leq n$} if condition (4) in \cref{prop:charcohdim} holds for $X$. See for example \cite{walsh2006dimension} and \cite{dranishnikov1990generalized} for more on the cohomological dimension of a compact metrizable space.
\end{remark}

\begin{lemma}\label[lemma]{lemma:finite-dimension-hypercomplete}
Let $R$ be a connective $\bE_{\infty}$-ring, and let $X$ be an LCH space with $\dim^!_R(X)<\infty$. Then $X$ is $R$-hypercomplete. 
\end{lemma}
\begin{proof}
Let $F\in\Sh{X}{\text{Mod}_R}$ be an $\infty$-connective sheaf. We need to show that $F = 0$. By covariant Verdier duality, it suffices to show that, for each $U$ open in $X$, $a_!j_!j^*F= 0$. For each open inclusion $j\colon U\hookrightarrow X$, the sheaf $\pfp{j}\pb{j}F$ again has trivial stalks and is thus $\infty$-connective. Since $\dim^!_R(X)<\infty$, we deduce that $a_!j_!j^*F\in (\Mod_R)_{\geq k}$ for all $k\geq 0$ so that $a_!j_!j^*F = 0$ as needed.
\end{proof}

\begin{lemma}\label[lemma]{lemma:cohsmoothandfin!cdim}
Let $R$ be a connective $\bE_{\infty}$-ring, and let $X$ be a $\Mod_R$-smooth LCH space. Then $\dim^!_R(X)<\infty$, and in particular, $X$ is $R$-hypercomplete. If we additionally assume that there exists $n$ such that, for all $x\in X$, we have $\pb{x}(\omega_X^R)\simeq\Sigma^nR$, then $\dim^!_R(X)=n$.
\end{lemma}
\begin{proof}
Let $a\colon X\rightarrow\ast$ be the unique map. By \cref{kunneth=>1ULA=loccontr} and \cref{1ULA=>loccontr}, we have a natural equivalence $\pfp{a}(-)\simeq\pfs{a}(-\otimes\omega_X^{-1})$. Since $R$-hypercompleteness can be checked locally and $\omega_X$ is locally constant, we can assume all stalks of $\omega_X$ equivalent. Moreover, they are connective by \cref{thomiso}.  Therefore, there exists $n\in\mathbb{N}$ such that, for any $x\in X$ we have $(\omega_X^{-1})_x\in(\text{Mod}_R)_{\geq -n}$. In particular, whenever $F\in \Sh{X}{\text{Mod}_R}_{\geq0}$, we find $F\otimes\omega_X^{-1}\in \Sh{X}{\text{Mod}_R}_{\geq -n}$. Since $\pfs{a}$ is a left adjoint to a $t$-exact functor, we deduce that $\pfp{a}F\simeq\pfs{a}(F\otimes\omega_X^{-1})$ belongs to $(\text{Mod}_R)_{\geq -n}$, and consequently that $\dim^!_R(X)\leq n <\infty$.
\end{proof}

We have just argued that the invertibility of $\omega_X$ implies, in particular, hypercompleteness of $X$. For this conclusion, however, it often suffices that $\omega_X$ is itself hypercomplete.

\begin{lemma}\label[lemma]{lemma:sufficient-hypercompleteness}
Let $X$ be a LCH space and $R$ and $\bE_\infty$-ring spectrum such that the functor $\homsp{\Mod_R}{-}{R} \colon \Mod_R^\mathrm{\op} \to \Mod_R$ is conservative.\footnote{Examples include certain PID's like the integers, fields, or ring spectra like $\mathrm{KU}$. Counterexamples include CDVR's like $\Z_p$ or ring spectra like $\mathbb{S}$. Indeed, $\hom_{\Z_p}(\mathbb{Q}_p,\Z_p) = 0 = \hom_{\mathbb{S}}(H\mathbb{F}_p,\mathbb{S})$.} If $\omega_X^R$ is $R$-hypercomplete, then $X$ is $R$-hypercomplete.
\end{lemma}
\begin{proof}
Let $F$ be an $\infty$-connective sheaf on $X$. By the same reasoning as in the proof of \cref{lemma:finite-dimension-hypercomplete}, it suffices to prove that $a_!j_!j^*(F) = 0$ for all open subsets $U \subseteq X$. Note that if $\omega_X$ is $R$-hypercomplete and $U \subseteq X$ is open, then $\omega_U=j^*(\omega_X)$ is also hypercomplete. It therefore in fact suffices to show that $a_!(F) = 0$. To that end, we have
\[ \homsp{\Mod_R}{a_!(F)}{R} = \homsp{\Sh{X}{\Mod_R}}{F}{\omega_X^R} = 0 \]
so that we conclude by assumption.
\end{proof}

\subsection{The dualizing spectrum of the shape}
Let $X$ be a $\Ani$-locally contractible CH space. In this subsection we show that, if the $R$-valued dualizing sheaf $\omega^R_X$ is locally constant, then $\sh(X)$ is a $R$-Poincar\'e duality complex, with $\omega^R_X$ tensor inverse to the Spivak normal fibration of $\sh(X)$. In particular, we deduce that whenever $X$ is additionally $R$-smooth, than $\sh(X)$ is a $R$-Poincar\'e duality complex.

We first recall the following general result, see also \cite[Appendix A]{Land-MJM}. Suppose $A$ is an anima and denote by $r$ the unique map $A \to \ast$. We may then consider for any $\bE_\infty$-ring spectrum $R$ the restriction functor 
\[ r^*\colon \Mod_R \to \Fun(A,\Mod_R)\]
which has a left adjoint (given concretely by forming the colimit of an $A$-shaped diagram) often written as $r_!$ and a right adjoint (given concretely by forming the limit of an $A$-shaped diagram) often written $r_*$. If $A$ is compact, then $r_!$ preserves limits and $r_*$ preserves colimits. It follows from Morita theory that there exists a unique object $D_A^R \in \Fun(A,\Mod_R)$ and an equivalence $r_*(-) \simeq r_!(- \otimes_R D_A^R)$ of $R$-linear functors. Moreover, $r_!$ admits a left adjoint $r^!$ which satisfies $r^!(R) = D_A^R$. In case $R=\bS$ is the sphere spectrum, we simply write $D_A$ instead of $D_A^\bS$. In general, one has $D_A^R = D_A^\bS \otimes R$. A compact anima $A$ is called an \emph{$R$-Poincar\'e duality complex} if $D_A^R$ is an invertible object of $\Fun(A,\Mod_R)$.

With these preliminaries fixed, let us now consider an LCH space $X$ which is $\Ani$-locally contractible, so that its shape $\sh(X)$ is an anima (rather than a pro-object in anima). First, we note that if $X$ is compact, then $\sh(X)$ is a compact anima. Indeed, recall that $\sh(X) = \pfs{a}(\pb{a}(\ast))$, that $\pfs{a}$ preserves compact objects because its right adjoint $\pb{a}$ admits a further right adjoint, and $a^*$ preserves compact objects because $\pf{a}$ preserves filtered colimits since $X$ is compact. Our main aim is to provide formulas for the dualizing spectrum $D_{\sh(X)}^R$ of $\sh(X)$ in terms of the dualizing object $\omega_X^R$ of $X$. 

To that end, we first recall that there is a canonical fully faithful and symmetric monoidal functor $\pb{\eta}\colon \Fun(\sh(X),\Mod_R) \hookrightarrow \Sh{X}{\Mod_R}$ whose image consists of $\Sh{X}{\Mod_R}^{\mathrm{lc}}$, the full subcategory of $\Sh{X}{\Mod_R}$ spanned by the locally constant sheaves. The functor $\iota$ admits both a right adjoint, denoted by $\pf{\eta}$ and a left adjoint, denoted by $\pfs{\eta}$. Moreover, we have the following commutative triangle:
\[\begin{tikzcd}
	& \Fun(\sh(X),\Mod_R) \ar[d,"\pb{\eta}"] \\
	\Mod_R \ar[r,"a^*"'] \ar[ur,"r^*"] & \Sh{X}{\Mod_R}
\end{tikzcd}\]
Passing to left and right adjoints of this diagram, we obtain the triangles
\[\begin{tikzcd}
	& \Fun(\sh(X),\Mod_R) \ar[dl,"r_!"'] && & \Fun(\sh(X),\Mod_R) \ar[dl,"r_*"']\\
	\Mod_R & \Sh{X}{\Mod_R} \ar[l,"\pfs{a}"] \ar[u,"\pfs{\eta}"'] && \Mod_R & \Sh{X}{\Mod_R} \ar[l,"\pf{a}"] \ar[u,"\pf{\eta}"']
\end{tikzcd}\]
i.e.\ we have equivalences $\pfs{a} = r_!\pfs{\eta}$ and $\pf{a} = r_*\pf{\eta}$. For more details, see \cite[Appendix A]{lurie2017higher}.

Let us now assume that $X$ is compact, so that $\pf{a}=\pfp{a}$ and $r_*=r_!(-\otimes D_{\sh(X)}^R)$.
Since $\pfs{\eta}\pb{\eta} \simeq \id \simeq \pf{\eta}\pb{\eta}$, we then obtain the following equivalences:
\[ \pf{a}(\pb{\eta}(-)\otimes\omega_X^R) \simeq \pfs{a}(\pb{\eta}(-)) = r_!(-) \quad \text{ and } \quad \pf{a}(\pb{\eta}(-)) \simeq r_!(-\otimes D_{\sh(X)}^R). \]
On the right hand equivalence, we may pass to right adjoints and obtain an equivalence of functors 
\[ \pf{\eta}\pbp{a} \simeq \sHom{\Fun(\sh(X),\Mod_R)}{D_{\sh(X)}^R}{r^*(-)} \]
which, when evaluated on the unit $R$ of $\Mod_R$ gives an equivalence
\[ \pf{\eta}\omega_X^R \simeq \sHom{\Fun(\sh(X),\Mod_R)}{D_{\sh(X)}^R}{R_X}.\]
This shows that the dualizing spectrum $D_{\sh(X)}^R$ determines $\pf{\eta}\omega_X^R$. Somewhat suprisingly, we were not able to show that $D_{\sh(X)}^R$ can be determined from $\omega_X^R$ in a similar manner in general. However, we have the following result, which is perfectly sufficient for all our applications.
\begin{theorem}\label[theorem]{thm:loccontrlocconstdualisPD}
	Let $X$ be a CH space which is $\Ani$-locally contractible and assume that $\omega_X^R$ is locally constant.
	Then $\sh(X)$ is an $R$-Poincar\'e duality complex and $D_{\sh(X)}^R = (\omega_X^R)^{-1}$.
\end{theorem}
\begin{proof}
	By the above, and since $\omega_X$ is locally constant and $\pb{\eta}$ is symmetric monoidal, we find 
	\[ r_!(-) = \pf{a}(\pb{\eta}(- \otimes \omega_X^R)) = r_!(- \otimes \omega_X^R \otimes D_{\sh(X)}^R) \]
	which implies that $\omega_X^R \otimes D_{\sh(X)}^R = R_X$.
\end{proof}

Recall that dualizable (and hence in particular invertible) objects of $\Sh{X}{\Mod_R}$ are precisely the locally constant sheaves with dualizable stalks \cite[Corollary 2.5.4.12]{martini2022presentable}. From this and \cref{thm:loccontrlocconstdualisPD} we immediately deduce the following corollary.

\begin{corollary}\label[corollary]{cptcohsmoothisPDcomplex}
	Let $X$ be $R$-smooth and $\Ani$-locally contractible CH space. Then $\sh(X)$ is an $R$-Poincar\'e duality complex whose dualizing spectrum is given by the inverse of the dualizing sheaf $\omega_X^R$.
\end{corollary} 

We interpret the above result that locally constancy of $\omega_X^R$ is sufficient for invertibility as an analog  of Klein's theorem \cite[Theorem A]{Klein} that if $D_A^R$ is dualizable for a compact anima $A$, then it is in fact invertible, see also \cite[Remark A.9]{Land-MJM}.

\section{Homology and homotopy manifolds}

In this section, we introduce \textit{$\C$-homology and homotopy manifolds}. Our main aim is to study conditions implying  $R$-smoothness for a LCH space. More specifically, we prove a spectral generalization of Wilder's local orientability conjecture, and even a new unstable version thereof. At the end of the section, we turn our attention to more geometric examples of homotopy manifolds, that we call \textit{homotopy manifolds with conical singularities}. We prove a generalization of a theorem of Siebenmann, showing that any homotopy manifolds with conical singularities is a topological manifold.

\subsection{Definition of homology and homotopy manifolds}

This subsection is devoted to defining appropriate sheaf theoretical notions of homology and homotopy manifolds. Before giving a proper definition, we start with the observation that, as a consequence of \cref{thomiso}, for nice enough topological spaces one has an unstable lift of the dualizing sheaf.

\begin{lemma}\label[lemma]{lemma:destabdualsheaf}
	Let $\D$ be any pointed presentable $\infty$-category, let $f\colon X \to Y$ be $\D$-locally contractible map between LCH spaces. Denote by $\Sigma^{\infty}\colon \D\rightarrow\D\otimes\Sp$ the canonical functor from $\D$ to its stabilization. Then there is a natural isomorphism
	$$\Sigma^{\infty}\pfs{(p_{1}^f)}\pf{(\Delta_f)}\pb{f}(-)\simeq\pbp{f}\Sigma^{\infty}(-)$$
	of functors $\D\rightarrow \Sh{X}{\D\otimes\spectra}$, where $\Delta_f\colon X \to X \times_Y X$ is the diagonal and $p_1^f \colon X \times_Y X\to X$ is the projection to the first factor.
\end{lemma}
\begin{proof}
Since all functors involved commute with colimits and are compatible with Lurie's tensor product, one may assume that $\C=\Ani_{\ast}$. The lemma then follows from the observation that $\Sigma^{\infty}$ commutes with all operations involved, combined with the description of $\pbp{f}$ provided by \cref{thomiso}, using that $\Delta_f$ is a closed immersion and hence proper.
\end{proof}

\begin{remark}
The conclusion of \cref{lemma:destabdualsheaf} still holds true if we relax the assumption of $X$ and $Y$ being LCH to the requirement that $f$ is separated and locally proper. See \cite[Appendix B]{Maegawa2024BauerFuruta} for a construction of the shriek operations along separated locally proper maps, building upon the results of \cite{MartiniWolf2025Proper}. 
\end{remark}

\begin{notation}\label[notation]{notation:relative-dualizing-sheaf}
Let $f\colon X \to Y$ be a map of topological spaces and let $\D$ be any $\infty$-category equipped with a closed symmetric monoidal structure. Assume that $f$ is $\D$-locally contractible. The \emph{$\D$-valued relative dualizing sheaf}, denoted by $\omega^{\D}_f$, is given by
\begin{enumerate}
	\item the sheaf $\pbp{f}_{\D}(1_Y)$, if $\D$ is stable and bicomplete;
	\item the sheaf $\pfs{(p_1^X)}\pf{(\Delta^X)}\pb{f}(1_Y)$, if $\D$ is pointed presentable.
\end{enumerate}
When $Y=\ast$, we will write $\omega_X^\D$ for short and refer to it as the dualizing sheaf of $X$, and when $\D=\Mod_R$ for some $\bE_{\infty}$-ring $R$, we write $\omega^{R}_f$ for short. Note that in case $\D$ is stable and presentable and $f$ is $\D$-locally contractible, the two definitions indeed agree, by \cref{thomiso}. 
\end{notation}

In case $\D$ is stable and bicomplete, using \cref{kunneth=>1ULA=loccontr} we have seen in \cref{unitULAdualizingsheafstableunderbc} that $\omega_f^\D$ is compatible with pullbacks. The same holds true in case $\D$ is pointed presentable:
\begin{lemma}\label[lemma]{lemma:pullback-unstable-dualizing-sheaf}
For a pullback diagram \eqref{pb} where $f$ is universally $\D$-locally contractible, the canonical map $v^*\omega_f^\D \to \omega_g^\D$ is invertible.
\end{lemma}
\begin{proof}
Each of the squares in the following diagram is a pullback square:
\[\begin{tikzcd}
	W \ar[r,"v"] \ar[d,"\Delta_g"] & X \ar[d,"\Delta_f"] & \\ 
	W \times_Z W \ar[d,"p_1^g"] \ar[r,"\bar{v}"] & X \times_Y X \ar[r] \ar[d,"p_1^f"] & X \ar[d,"f"] \\
	W \ar[r,"v"] & X \ar[r,"f"] & Y
\end{tikzcd}\]
Therefore, we may use that $\pfs{(-)}$ and $\pf{(-)}$ are compatible with pullbacks in the present situation as $p_1^f$ is $\D$-locally contractible and $\Delta_f$ is proper:
\begin{align*}
	v^*\omega_f^\D & = v^*\pfs{(p_1^f)}\pf{(\Delta_f)}(1) \\
		& = \pfs{(p_1^g)}\bar{v}^*\pf{(\Delta_f)}(1) \\
		& = \pfs{(p_1^g)} \pf{(\Delta_g)} v^*(1) = \omega_g
\end{align*}
as claimed.
\end{proof}

Next, we give a concrete description of the stalks of the dualizing sheaf.
\begin{notation}
For $X$ a $\D$-locally contractible LCH space and $x \in X$, we will denote by $\sh_\D(X | x)$ the cofibre of the canonical map $\sh_\D(X\setminus\{x \}) \to \sh_\D(X)$ and call it the local shape of $X$ at $x$.
\end{notation}
\begin{lemma}\label[lemma]{lemma:stalkdualizingshapeatx}
Let $\D$ be any stable bicomplete, or pointed presentable, $\infty$-category equipped with a closed symmetric monoidal structure and $X$ a $\D$-locally contractible LCH space. Then the stalk of $\omega_X^\D$ at $x \in X$ is given by the local shape $\sh_\D(X|x)$ of $X$ at $x$.
\end{lemma}
\begin{proof}
We treat the stable case first.
We denote by $i$ the closed inclusion $\{x\} \to X$, by $j$ the open inclusion $X \setminus \{x\} \to X$, by $a\colon X \to \ast$ the unique map and by $b$ the composite $aj$.
By \cite[Corollary 4.7]{volpe2023operations}, the following is a cofibre sequence of functors.
\[ \pfp{j}\pbp{j} \to \id \to \pf{i}\pb{i}.\]
Applying $a_!$ from the left and $a^!$ from the right, and using again that $i_* = i_!$ we obtain a cofibre sequence
\[ \pfp{b}\pbp{b} \to \pfp{a}\pbp{a} \to \pb{i}\pbp{a} \]
which yields the claim by \cref{sharp=shriektensdualwithkunnet}. In the pointed presentable case, we use \cref{thomiso} and the denote by $j$ the open inclusion of $X \times X \setminus \Delta(X) \to X\times X$. By \cite[Cor.\ 4.7]{volpe2023operations} we then find that $\Delta_*\Delta^*$ is a cofibre, from which, using base-change, the claim follows again.
\end{proof}

\begin{definition}\label[definition]{Def:homology-manifold}
Let $\D$ be any stable bicomplete, or pointed presentable, $\infty$-category equipped with a closed symmetric monoidal structure, and write $1$ for the unit in $\D$. A $\D$-locally contractible LCH space $X$ is called a \emph{$\D$-homology manifold of dimension $n$} if for all $x \in X$, the stalk, or equivalently the local shape $x^*\omega_X^\D \simeq \sh_\D(X|x)$ at $x$ is equivalent to $\Sigma^n 1$. 
When $\D=\Mod_R$ for some $\bE_\infty$-ring spectrum $R$, we use the term \emph{$R$-homology manifold} rather than $\Mod_R$-homology manifold and when $\D=\Ani_{\ast}$, we use the term \emph{$\Ani_*$-homotopy manifold} rather than $\Ani_*$-homology manifold.
\end{definition}

We emphasize that in this definition, we do \emph{not} require $X$ to be $\D$-locally weakly contractible, and that being an $R$-homology manifold is a condition on the $R$-homology of the local \emph{shapes} of $X$ at points of $X$, \emph{not} on the $R$-homology of the local singular complexes, i.e.\ simply the cofibre of the map on underlying weak homotopy types of the inclusions $X \setminus \{x\} \to X$. Said yet differently, an $R$-homology manifold requires the local sheaf homology groups to be concentrated in a single degree, rather than the local singular homology groups, and we do not want to impose conditions on $X$ which ensure that these two notions agree.

\begin{remark}\label[remark]{rmk:homotopymannotgriffiths}
We warn the reader that our definition of $\Ani_*$-homotopy manifolds differs from the homotopy manifolds introduced by Griffiths \cite{Griffiths}, and later on considered also by Lacher, Curtis, Wilder and others (see for example \cite[Section 4]{lacher1969cell}). At present, the precise relationship between our definition and Griffiths's is unclear to us.
\end{remark}

\begin{remark}\label[remark]{PIDhommanonpi0}
Suppose $R$ is an $\bE_\infty$-ring such that every invertible $R$-module is of the form $\Sigma^n R$ for some $n\in \Z$, e.g.\ if $R$ is connective and $\pi_0(R)$ is a principal ideal domain. Then a $R$-locally contractible LCH space $X$ is an $R$-homology manifold if and only if all stalks $x^*(\omega_X^R)$ are invertible $R$-modules. In fact, by Hurewicz's theorem, it suffices to assume that all stalks of $\omega_X^{\pi_0(R)}$ are invertible $\pi_0(R)$-modules.
\end{remark}

In Subsections~\ref{subsec:Bredon-Wilder},~\ref{subsec:Wilder}, and~\ref{subsec:Wilder2} and we will discuss the relation of $R$-homology manifolds with Wilder's generalized manifolds, Wilder's local orientability conjecture, and Wilder's monotone mapping theorem. The final two subsections of this paper are then devoted to the study of $\Ani_*$-homotopy manifolds.

\subsection{Comparison with Bredon and Wilder}\label{subsec:Bredon-Wilder}
First, we explain some relation of the above notion with Wilder's generalized manifolds in the sense of Bredon \cite{Bredon}. This is defined relative to a fixed PID $R$. In our notation, a Wilder $n$-manifold over $R$ is a LCH space $X$ such that
\begin{enumerate}
\item $\dim^!_R(X) <\infty$,
\item The stalks $x^*(\omega_X^R)$ are isomorphic to $\Sigma^n K$, and
\item It is $clc^\infty_R$\footnote{$clc$ stands for cohomologically locally connected.}, i.e.\ for all $U \subseteq K \subseteq V$, with $U,V$ open and $K$ compact, and all $n\geq 0$, the induced map $H^n_c(V;R) \to H^n_c(U;R)$ has finitely generated image, see \cite[Prop.\ 17.2]{bredon2012sheaf}.
\end{enumerate}

\begin{remark}\label[remark]{rmk:clcRtoresidue}
	It follows from the universal coefficient theorem that for any maximal ideal $\mathfrak{m} \subseteq R$ with residue field $K=R/\mathfrak{m}$, a Wilder $n$-manifold over $R$ is also a Wilder $n$-manifold over $K$. Indeed, for any $R$-module $M$, we have $a_!(X) \otimes_R M = a_!(X \otimes_R a^*(M))$ and therefore we have a short exact sequence
	\[ 0 \to H^n_c(X;R) \otimes_R M \to H^n_c(X;R) \to \Tor_1^R(H^{n+1}_c(X;R),M) \to 0 \]
	since $R$ is 1-dimensional (compare with \cite[Theorem 15.3]{bredon2012sheaf}).
\end{remark} 

\begin{lemma}\label[lemma]{lemma:clcooandloccontrK}
Let $K$ be a field and $X$ a LCH Hausdorff space. If $X$ is $K$-locally contractible, then it is $clc^\infty_K$. Conversely, if $\dim^!_K(X)<\infty$ and $clc^\infty_K$, then $X$ is $K$-locally contractible.
\end{lemma}
\begin{proof}
We recall from \cite{KN} that $\Sh{X}{\Mod_K}$ is a dualizable presentable stable $\infty$-category whose compact morphisms are generated, as an ideal, from the morphisms of sheaves associated to open inclusions $U \to V$ that factor through a compact subset of $X$ (see \cite[Lemma 4.4.13]{KN}). Moreover, in a compactly generated $\infty$-category, compact morphisms are precisely those that factor over a compact object \cite[Example 2.2.6 (2)]{KN}. Hence, if $\dim^!_K(X)<\infty$, and since $K$ is a field, the condition $clc^\infty_K$ is equivalent to the condition that $\cSec{U}{K} \to \cSec{V}{K}$ factors through a perfect complex, i.e. it is a compact morphism. In particular, $a_!$ preserves compact morphisms, which is the case if and only if its right adjoint $a^!$ admits a right adjoint. By covariant Verdier duality, this is the case if and only if $a^*$ admits a left adjoint, which means that $X$ is $K$-locally contractible.
\end{proof}

\begin{corollary}
	Let $K$ be a field, and let $X$ be a Wilder $n$-manifold over $K$. Then $X$ is a $K$-homology manifold of dimension n.
\end{corollary}
\begin{proof}
	By \cref{lemma:clcooandloccontrK}, condition (1) and (3) in the definition of a Wilder $n$-manifold imply that $X$ is $K$-locally contractible. Requiring additionally condition (2) gives exactly the definition of a $K$-homology manifold.
\end{proof}

In particular, Wilder's generalized manifolds are both more general than our $R$-homology manifolds (they need not be $R$-locally contractible if $R$ is not a field) and less general (they are assumed to satisfy $\dim^!(-)<\infty$ and are hence hypercomplete). 

\subsection{Wilder's local orientability conjecture}\label{subsec:Wilder}

Let $R$ be a connective $\bE_{\infty}$-ring spectrum. In this subsection, we focus on showing that, under appropriate finite dimensionality assumptions, $R$-homology manifolds are $R$-smooth. This provides a version of Wilder's local contractibility conjecture with coefficients in $R$. As a consequence, we'll see that ANR homology manifolds with finite $\mathbb{F}_p$-dimension are $\spectra$-smooth.

We first recall the following immediate consequence of \cite[Corollary 2.5.4.12]{martini2022presentable}.
\begin{lemma}\label[lemma]{charinvertsheaves}
Let $X$ be a topological space and $R$ be an $\bE_{\infty}$-ring spectrum. A sheaf $F\in\Sh{X}{\Mod_R}$ is invertible if and only if it is locally constant with invertible stalks.
\end{lemma}

We also need the following lemma about the preservation of cohomological smoothness under extension of scalars.

\begin{lemma}\label[lemma]{cohsmoothaftertensoring}
	Let $\D\rightarrow\D'$ be any symmetric monoidal functor between stable and bicomplete $\infty$-categories equipped with closed symmetric monoidal structures and $X$ an LCH space. If $X$ is $\D$-smooth, it is also $\D'$-smooth. 
\end{lemma}
\begin{proof}
	First, by \cite[Corollary 5.16]{volpe2023operations} we have $\Sh{X}{\D'}\simeq\Sh{X}{\D}\otimes_{\D}\D'$.
	We then note that $a^*_{\D'} = a^*_\D \otimes_\D \D'$ and since $\omega^\D$ is invertible, $a^!_\D = a^*_\D \otimes \omega_\D$ is $\D$-linear and cocontinuous. It follows that $a^!_{\D'} = a^!_\D \otimes_\D \D'$ as well, and hence, that $\omega_X^{\D'} = \omega_X^\D \otimes_\D \D'$ is invertible. Moreover, $\pfs{a}^\D \otimes_\D \D'$ is left adjoint to $a^*_{\D'}$, so $X$ is indeed $\D'$-smooth by \cref{kunneth=>1ULA=loccontr}.
\end{proof}

The proof of the following theorem closely follows the argument of \cite[Proposition 7.9]{scholze2022six} and, in fact, also Bredon's proof of Wilder's local orientability conjecture \cite{Bredon}. In what follows, for a connective $\bE_\infty$-ring, we shall refer to its residue fields as the rings $K=\pi_0(R)/\mathfrak{m}$ where $\mathfrak{m} \subseteq \pi_0(R)$ is a maximal ideal.
\begin{theorem}\label[theorem]{prop:locally-constant}
Let $R$ be a connective $\bE_\infty$-ring and $X$ a $R$-locally contractible LCH space. Assume that $X$ is $K$-hypercomplete for all residue fields $K$ of $R$.\footnote{Equivalently, that $\omega_X^K$ is $K$-hypercomplete by \cref{lemma:sufficient-hypercompleteness}.}
Then $\omega_X^R$ is invertible if and only if for all $x \in X$, its stalk $x^*(\omega_X^{\pi_0(R)})$ is invertible and for all $y,z\in X$ lying in the same connected component, we have $y^*(\omega_X^R)\simeq z^*(\omega_X^R)$.
\end{theorem}
\begin{proof}
The ``only if'' is immediate as the functors $x^*$ are symmetric monoidal. Let us prove the converse and write $L = x^*\omega_X^R$. Since invertible $R$-modules are compact in $\Mod_R$, there is an open neighborhood $U$ of $x$ and a map $s\colon L_U \to (\omega_X^R)_{|U}=\omega_U^R$ inducing the previous equivalence on the stalk at $x$; here $L_U$ denotes the constant sheaf. By \cref{Sploccontr=>Setlocconn} and \cref{toposlocconn=locconn}, we know that $X$ is locally connected, and therefore we can assume that $U$ is connected. We aim to show that $s$ is an equivalence and hence that $\omega_X^R$ is invertible by \cref{charinvertsheaves}. We claim that it suffices to show that $s$ is an equivalence on stalks. Indeed, the fibre $F$ of $s$ is $\infty$-connective, so since $L_U$ is hypercomplete by \cref{loccontr=>locconstishyp}, we deduce that $\omega_U^R \simeq L_U \oplus \Sigma F$. By \cref{1ULA=>loccontr} we have $\sHom{U}{\omega^{R}_U}{\omega^{R}_U}\simeq R_U$. Since $U$ is connected, it follows that $F=0$ as claimed. 

 Now we show that $s$ induces an isomorphism on stalks.
So let $y\in U$ and let $y^*(s)\colon L \to L$ be the map induced by $s$ on stalks at $y$. We want to show that $y^*(s)$ is an isomorphism. Since $L$ is connective by \cref{lemma:stalkdualizingshapeatx}, it suffices to check this after base-change to $\pi_0(R)$. Since $L\otimes_R\pi_0(R)$ is assumed to be an invertible $\pi_0(R)$-module, we may check this after base-change along $R \to S$ where $S= \pi_0(R)/\mathfrak{m}$ and $\mathfrak{m}$ is a maximal ideal. In other words, we may assume that $R=K$ is a field. 
Now denote by $Z$ the support of $s$, i.e.\ the points $z \in U$ such that $z^*(s) \neq 0$. Note that $Z$ is closed and non-empty since $x\in Z$. We denote by $i$ the inclusion $Z \to U$ and by $j$ the inclusion of the open complement $U \setminus Z \to U$. For any point $y\in U\setminus Z$, the map $\pb{y}(s):L\to\pb{y}(\omega^K_U)$ is, up to shift, an endomorphism of $K$. By definition of $Z$, $\pb{y}(s)$ is not invertible, and therefore $0$ since $K$ is a field. Observe that, by assumption, $\omega^K_U$ is hypercomplete and has isomorphic invertible stalks, and hence, up to shift, lies in the heart of $\Sh{X}{\Mod_K}$. 
Thus, we find that the composite $j_!j^*(L_U) \to L_U \to \omega_U^K$ is canonically trivialized, so that $s$ factors as the composite $L_U \to i_*i^*(L_U) =i_*L_Z \to \omega_U^K$ in which the latter map induces an isomorphism on stalks for points in $Z$. 
Consequently, the composition $i_*L_Z \to \omega_U^K\to\pf{i}\pb{i}\omega_U^K$ is an isomorphism, so that we get an isomorphism $i_*(L_Z) \oplus j_!j^*(\omega_U^K) \cong \omega_U^K$. 
Similarly as before, since $U$ is connected, $\sHom{U}{\omega_U^K}{\omega_U^K} = K_U$ and $i_*(L_Z)$ is non-trivial, we deduce that $j_!j^*(\omega_U^R)$ is trivial and hence that $Z=U$.
\end{proof}

\begin{remark}\label[remark]{rmk:wilder=assumptioninprop}
    Notice that, when $R=K$ a field, the requirements in the hypothesis of \cref{prop:locally-constant} are a priori less restrictive than those in the definition of a Wilder manifold over $K$, as we only demand $K$-hypercompleteness as opposed to finiteness of cohomological dimension. However, as a consenquence of \cref{prop:locally-constant} and \cref{lemma:cohsmoothandfin!cdim}, it turns out that the assumptions of \cref{prop:locally-constant} are in fact equivalent to being a Wilder manifold over $K$.
\end{remark}

\begin{remark}
	The above proof works in essentially the same way if we assume that $X$ is $R$-hypercomplete and is $K$-locally contractible for all residue fields $K$ of $R$, in particular for Wilder manifolds over $R$. In particular, it generalizes simultaneously \cite[Proposition 7.9]{scholze2022six} and \cite{bredon2012sheaf}. Indeed, in the above proof, we have used that $X$ is $R$-locally contractible only in two situations: Once to deduce that $X$ is locally connected. For this, it would be sufficient to assume that $X$ is $K$-locally contractible for some residue field $K$ of $R$. 
	
	Moreover, when we have argued that it suffices to show that the map $s$ is an equivalence on stalks, we have used that $\sHom{U}{\omega_U^R}{\omega_U^R} \simeq R_U$ which follows if $X$ is $R$-locally contractible. However, if instead we assume that $X$ is $R$-hypercomplete (for instance if $\dim^!_R(X)<\infty$) it again suffices to show that $s$ is an equivalence on stalks. 
	
	Finally, we have used again the equivalence $\sHom{U}{\omega_U^K}{\omega_U^K} = K_U$ for which, as indicated above, it suffices that $X$ is $K$-locally contractible.
\end{remark}

We observe that \cref{prop:locally-constant} gives the following characterization of $R$-smoothness.

\begin{corollary}\label[corollary]{cohsmoothcharPID}
	Let $R$ be a connective $\bE_\infty$-ring with $\pi_0(R)$ a PID, and $X$ be a $R$-locally contractible LCH space, which is $K$-hypercomplete for all $K=\pi_0(R)/\mathfrak{m}$ as above. Then $X$ is $R$-smooth if and only if $X$ is a $\pi_0(R)$-homology manifold. Moreover, if $R=K$ a field, then $X$ is $K$-smooth if and only if it is a Wilder manifold over $K$.
\end{corollary}
\begin{proof}
	Combine \cref{PIDhommanonpi0} and \cref{prop:locally-constant}. The last part of the corollary follows by \cref{rmk:wilder=assumptioninprop}.
\end{proof}

We can finally prove $\spectra$-smoothness for many topological spaces of interest in geometric topology.

\begin{theorem}
	Let $X$ be any locally compact ANR $\mathbb{Z}$-homology manifold $X$ such that, for any prime $p$, $X$ is $\mathbb{F}_p$-hypercomplete. Then $X$ is $\spectra$-smooth. When $X$ is additionally compact, we have that $\Sing(X)$ is a Poincar\'e duality complex. Moreover, the inverse of of the Spivak normal fibration of $\Sing(X)$ is given by the dualizing sheaf of $X$. 
\end{theorem}
\begin{proof}
	This follows combining \cref{LCANRloccontrlocsingshape}, \cref{cohsmoothcharPID} and \cref{cptcohsmoothisPDcomplex}.
\end{proof}

\subsection{Wilder's monotone mapping theorem}\label{subsec:Wilder2}

We now turn to studying how cohomological smoothness behaves under pushforward along \textit{$R$-cell-like maps}. Thanks to \cref{prop:locally-constant}, we'll deduce that the $\s{}$-cell-like image of a $\spectra$-smooth LCH space is $\spectra$-smooth.

\begin{definition}
	Let $\D$ be an $\infty$-category. A continuous and proper map $f\colon X\rightarrow Y$ between LCH spaces is called \emph{$\D$-cell-like} if $\pb{f}_{\D}\colon \Sh{Y}{\D}\rightarrow\Sh{X}{\D}$ is fully faithful. When $\D= \text{Mod}_R$ for a ring spectrum $R$, say \emph{$R$-cell-like} rather than $\Mod_R$-cell-like.
\end{definition}

\begin{remark}
	A map of topological spaces $f\colon X\rightarrow Y$ if $\Ani$-cell-like if and only if the induced geometric morphism $f_*\colon \Sh{X}{\Ani}\rightarrow\Sh{Y}{\Ani}$ is cell-like in the sense of \cite[Definition 7.3.6.1]{lurie2009higher}. When $X$ and $Y$ are both $\Ani$-locally contractible, $f$ is $\Ani$-cell-like if and only if it is a cell-like mapping in the sense of Lacher \cite{lacher1969cell}. For a proof, we refer to \cite{kremer2025approximate}.
\end{remark}

\begin{lemma}\label[lemma]{lemma:vietoris}
	Let $f\colon X \to Y$ be a proper map between LCH spaces, let $\D$ be a stable bicomplete $\infty$-category equipped with symmetric monoidal structure, and let $R$ be a ring spectrum such that the constant sheaves $R_X$ and $R_Y$ are hypercomplete. Then
	\begin{enumerate}
		\item $f$ is $\D$-cell-like if and only if $\pbp{f}$ is fully faithful,
		\item $f$ is $\D$-cell-like if and only if the map $1_Y \to \pf{f}(1_X)$ is an equivalence, and
		\item $f$ is $R$-cell-like if and only if for each $y\in Y$, the map $R\rightarrow \Sec{f^{-1}(y)}{R_X}$ is invertible.
	\end{enumerate}
\end{lemma}
\begin{proof}
	In general, given a functor $\varphi$ with left adjoint $\psi$ and right adjoint $\vartheta$, we have that $\psi$ is fully faithful if and only if $\vartheta$ is fully faithful, and this is the case precisely if $\varphi$ is a Dwyer-Kan localization. This shows (1).
	By the projection formula, the unit map $\eta_F \colon F \to \pf{f}\pb{f}(F)$ is equivalent to the map $F \otimes \eta_{1_Y} \colon F \otimes 1_Y \to F \otimes \pf{f}(1_X)$. Consequently, $\eta_F$ is an equivalence for all $F$ if and only if $\eta_{1_Y}$ is an equivalence, so that (2) follows. To see (3), we note that $\pf{f}$ preserves hypercomplete sheaves (since its left adjoint is $t$-exact and hence preserves $\infty$-connective sheaves). In particular, to see that $R_Y \to \pf{f}(R_X)$ is an isomorphism, it suffices to know that the induced map on stalks are isomorphisms. By base-change, these maps are given by the maps $R \to \Sec{f^{-1}(y)}{R_X}$ so the lemma follows.
\end{proof}

\begin{remark}
	When $R$ is a discrete ring, the constant sheaves $R_X$ and $R_Y$ are hypercomplete. \cref{lemma:vietoris}(3) hence implies that $f$ is $R$-cell-like if and only if it is a Vietoris map in the sense of \cite[Theorem 16.33]{bredon2012sheaf}.
\end{remark}

We now aim to prove a generalization of \cite[Theorem 16.33]{bredon2012sheaf}. First, we have:

\begin{lemma}\label[lemma]{loccontrundercelllikeimage}
	Let $\D$ be an $\infty$-category which is either presentable or stable and bicomplete. Let $f\colon X\rightarrow Y$ be a $\D$-cell-like map. If $X$ is $\D$-locally contractible then $Y$ is also $\D$-locally contractible.
\end{lemma}

\begin{proof}
	Let us denote by $a \colon X \to \ast$ and $b\colon Y \to \ast$ the unique maps. Then $\pb{b} \to \pf{f}\pb{f}\pb{b} = \pf{f}\pb{a}$ is an equivalence since $\pb{f}$ is fully faithful. Hence, $\pb{b}$ admits a left adjoint if $\pb{a}$ does, showing the desired conclusion.
\end{proof}

\begin{theorem}[Wilder's monotone mapping theorem]\label[theorem]{thm:wildermonotone}
	Let $R$ be a connective $\bE_{\infty}$-ring, and let $f\colon X\rightarrow Y$ be an $R$-cell-like map, where $X$ is $R$-smooth. Suppose that there exists a $\otimes$-invertible $R$-module $M$ such that one of the following two conditions hold
	\begin{enumerate}
		\item $\omega_X^R \simeq M_X$ is equivalent to the constant sheaf on $M$;
		\item $\pi_0(R)\cong\mathbb{Z}$ and all stalks of $\omega_X^{\mathbb{Z}}$ are isomorphic to $M\otimes_{R}\mathbb{Z}$.
	\end{enumerate}
    Then $Y$ is $R$-smooth.
\end{theorem}
\begin{proof}
Using \cref{loccontrundercelllikeimage}, we already know that $Y$ is $R$-locally contractible. Therefore, by \cref{prop:locally-constant} and \cref{lemma:sufficient-hypercompleteness}, it suffices to show that $\omega_Y^K$ is hypercomplete for any residue field $K$ of $R$, and that the stalks of $\omega_Y^{R}$ are invertible. 
	
By \cref{lemma:vietoris}(1) $\pbp{f}$ is fully faithful. In particular, the map $\pfp{f}(\omega_X^K) \simeq \pfp{f}(\pbp{f}(\omega_Y^K)) \to \omega_Y^K$ is an isomorphism. As above, $\pfp{f} = \pf{f}$ preserves hypercomplete sheaves, showing that $\omega_Y^K$ is hypercomplete.

We first deal with the case when $\omega_X^{R}$ is constant. Fix $y\in Y$. Since $f$ is $R$-cell-like and $\omega^R_X$ is constant, by taking stalks at $y\in Y$ we obtain equivalences
\[M\simeq\Sec{f^{-1}(y)}{M}\simeq\Sec{f^{-1}(y)}{\omega^R_X}\simeq(\pf{f}\omega^R_X)_y\simeq(\omega^R_Y)_y.\]
Therefore, we deduce that $(\omega^R_Y)_y$ is invertible.
	
Let us now consider the case $\pi_0(R)=\Z$. By assumption, there exists $n\in\mathbb{N}$ such that $\Omega^n\omega^{\Z}_X$ is locally constant with stalks isomorphic to $\Z$. Each fiber of a $\Z$-cell-like map is $M$-acyclic for any abelian group $M$, so $H^1_{\mathrm{sheaf}}(f^{-1}(y), \Z/2)= 0$. Since this group classifies locally constant sheaves on $f^{-1}(y)$ with infinite cyclic stalks, we deduce that $\Omega^n(\omega^{\Z}_X)_{|f^{-1}(y)}$ is constant. The same argument as in the previous case then shows that for any $y\in Y$, the stalk $(\omega^{\mathbb{Z}}_Y)_y$ is invertible. Hence we can conclude by \cref{cohsmoothcharPID}.
\end{proof}

\subsection{Homotopy manifolds}

In this section we focus on pointed unstable coefficients.
Our first result is an analogue of \cref{prop:locally-constant} with coefficients in $\Ani_{\ast}$.

\begin{theorem}\label[theorem]{thm:unstablelocalorient}
    Let $X$ be a hypercomplete $\Ani_*$-homotopy manifold. Then $\omega_X^{\Ani_{\ast}}$ is locally constant.
\end{theorem}
\begin{proof}
We argue similarly as in \cref{prop:locally-constant}. Fix a point $x\in X$, so that we have an isomorphism $S^n\simeq (\omega_X^{\Ani_{\ast}})_x$. By compactness of $S^n$, the isomorphism above extends to a local section $s: S^n_U\rightarrow\omega_U^{\Ani_{\ast}}$, and by \cref{toposlocconn=locconn} we can assume that $U$ is connected. We want to show that $s$ is an isomorphism. Since $X$ is assumed to be hypercomplete, it suffices to show that for any $y\in U$, the induced map of pointed animae 
\[s_y\colon S^n\rightarrow (\omega_U^{\Ani_{\ast}})_y\simeq S^n\]
    is invertible. The invertibility of an endomorphism of $S^n$ is determined by the degree, and therefore it suffices to show that $s_y$ induces an isomorphism after applying $\Sigma^{\infty}$. 

    Notice that since $X$ is a $\Ani_*$-homotopy manifold, it is in particular an $\Sp$-homology manifold. Hence we deduce by the proof of \cref{prop:locally-constant} that any local section on a connected open extending the isomorphism $\Sigma^n\s{}\simeq(\omega_X^{\Sp})_x$ must be invertible. In particular, we deduce that $\Sigma^{\infty}(s_y)$ must be invertible for all $y\in U$, and so our proof is concluded.
\end{proof}

We now observe that many $\Sp$-homology manifolds are in fact $\Ani_*$-homotopy manifolds. 

\begin{proposition}\label[proposition]{homology=>htpyman}
\label[proposition]{prop:unstable-fibration}
    Let $X$ be a locally compact Hausdorff space which is $\Ani$-locally contractible. Assume that $X$ is a $\Sp$-homology manifold of dimension $>1$. Then $X$ is a $\Ani_*$-homotopy manifold.
\end{proposition}
\begin{proof}
    It suffices to show that $\sh(X|x)$ is simply connected for all $x\in X$. Let $j\colon X\setminus \{x\}\hookrightarrow X$ and $i\colon\{x\}\hookrightarrow X$ be the inclusions, and let $a\colon X\to\ast$ be the unique map. Write $\mathrm{holink}(x,X)$ for the pro-anima corepresenting the left exact accessible functor $\pb{i}\pf{j}\pb{j}\pb{a}\colon\Ani\to\Ani$. Explicitly, this is given by $\text{``}\varprojlim\limits_{x\in U}\text{''}\sh(U\setminus\{x\})$. 
    
    We claim that there is a pushout square
    \begin{equation}\label{eq:squareholink}
    \begin{tikzcd}
    	\mathrm{holink}(x,X)\ar[r]\ar[d] & \sh(X\setminus \{x\})\ar[d] \\
    	\ast\ar[r] & \sh(X) 
    \end{tikzcd}
    \end{equation}
    in $\Pro(\Ani)$. To see this, one observes that, after passing to the respective corepresentable functors, \eqref{eq:squareholink} is obtained from the pullback square 
    $$
    \begin{tikzcd}
    	\id\ar[r]\ar[d] & \pf{j}\pb{j}\ar[d] \\
    	\pf{i}\pb{i} \ar[r] & \pf{i}\pb{i}\pf{j}\pb{j}
    \end{tikzcd}
    $$
    in $\Fun(\Sh{X}{\Ani}, \Sh{X}{\Ani})$ by precomposing with $\pb{a}$ and postcomposing with $\pf{a}$. Passing to vertical cofibers, we get an equivalence $$\Sigma(\mathrm{holink}(x,X))\simeq\sh(X|x)$$ in $\Pro(\Ani)$.
    
    To conclude the proof we only need to show that there is a cofinal system $\mathcal{F}$ of open neighbourhoods of $x$ with the property that, for any $V\in\mathcal{F}$, $\sh(V\setminus\{x\})$ is connected. Indeed, if this were the case, since finite colimits in $\Pro(\Ani)$ commute with cofiltered limits, $\sh(X|x)$ would be a retract of $\Sigma(\sh(V\setminus\{x\}))$ for some $V\in\mathcal{F}$. But since $V\in\mathcal{F}$, $\Sigma(\sh(V\setminus\{x\}))$ is simply connected, and therefore so would be $\sh(X|x)$. 
    
    We claim that $\sh(U\setminus\{x\})$ is connected whenever $U$ is a connected open neighbourhood of $x$. Indeed, since $X$ is a $\Sp$-homology manifold of dimension greater than $1$, using excision we get an isomorphism $H_1(\sh(X|x),\mathbb{Z})\cong H_1(\sh(U|x),\mathbb{Z})\cong 0$. Therefore, by the long exact sequence in homology, we see that $H_0(\sh(U\setminus \{x\}),\mathbb{Z})\to H_0(\sh(U),\mathbb{Z})$ is injective. Thus we deduce that $\sh(U\setminus \{x\})$ must be connected by connectedness of $\sh(U)$. The proof is finished by observing that, since $X$ is $\Ani$-locally contractible, it is also locally connected by \cref{toposlocconn=locconn}, and so we may pick $\mathcal{F}$ to be the family of connected open neighbourhoods of $x$.
\end{proof}

\begin{remark}
The argument of \cref{homology=>htpyman} does not work if the dimension of $X$ is $\leq 1$. Nevertheless, it is natural to wonder whether the concusion of \cref{homology=>htpyman} still holds in this case.
    
    If $X$ is an $\Sp$-homology manifold of dimension $0$, then it is automatically discrete, and hence a topological manifold. Indeed, it is in particular locally connected, so for each $x\in X$, the connected component $U$ where $x$ is contained has the property that the $0$-th homology group of $U\setminus\{x\}$ is isomorphic to $0$, i.e. $U\setminus\{x\}=\emptyset$. If $X$ is of dimension $1$, and it is additionally required to be metrizable, then one can still show that $X$ is a topological manifold (see \cite[Theorem 16.32]{bredon2012sheaf}). We do not know if the conclusion of \cref{homology=>htpyman} holds if we do not require $X$ to be metrizable.
\end{remark}

Next, we discuss to what extent the unstable dualizing sheaf exists in families. For $X$ a hypercomplete $\Ani_*$-homotopy manifold, let us denote by $\Baut^{\omega_X^{\Ani_*}}(X)$ the full subgroupoid of $\Ani_{/\Baut_*(S^d)} \subseteq \Ani_{/\Ani_*}$ containing  $(X,\omega_X^{\Ani_*})$. 

\begin{theorem}\label[theorem]{thm:functorialunstabledual}
	Let $X$ be a hypercomplete $\Ani_*$-homotopy manifold. Then the unstable dualizing sheaf is functorial in homeomorphisms, i.e. there exists a map in $\Ani$
	\[\BHomeo(X)\to \Baut^{\omega_X^{\Ani_*}}(X).\]
\end{theorem}
\begin{proof}
	It suffices to prove that the arrow exists in the homotopy category of $\Ani$, for which it suffices to argue that the arrow  exists after applying the Yoneda embedding. Since the homotopy category of $\Ani$ is the 1-categorical localization of the 1-category CW complexes at homotopy equivalences, it suffices to now show that the arrow exists after considering homotopy classes of maps from CW complexes $B$, natural in $B$. To proceed, we recall what the functors are that the objects in the displayed map represent. Namely, $\BHomeo(X)$ represents  fibre bundles $E \to B$ with fibre homeomorphic to $X$, and $\Baut^{\omega_X^{\Ani_*}}(X)$ represents fibrations $E \to B$ together with a functor $\omega\colon E \to \Ani_*$ whose restriction to any fibre is equivalent to $\omega_X^{\Ani_*}$ (in the slice $\Ani_{/\Ani_*}$).
	
	To construct the desired arrow, we associate to an $X$-fibre bundle $p \colon E \to B$ the underlying fibration of $p$  by applying the functor $\sh$\footnote{Note that, since $X$ is in particular $\Ani$-locally contractible, for any $X$-fibre bundle $E\to B$ with $B$ a CW-complex, the space $E$ is $\Ani$-locally contractible, and in particular $\sh(E)\in\Ani$ and the monodromy equivalence holds for $E$. We also warn the reader that, by a slight abuse of notation, we henceforth suppress the use of $\sh$ throughout the remainder of this proof.} together with the relative dualizing sheaf $\omega_p^{\Ani_*}$ from \cref{notation:relative-dualizing-sheaf}. To see that this is well-defined, we first note that $\omega_p^{\Ani_*}$ is a locally constant sheaf (and hence equivalently described by a functor $E \to \Ani_*$): Indeed, it suffices to show this locally on $B$, but since $p$ is locally trivial, we see that it suffices to show the claim in case $p$ is globally trivial, i.e.\ a projection $p_2\colon X \times B \to B$. In this case, $\omega_{p_2}^{\Ani_*}$ is given by $p_1^*(\omega_X^{\Ani_*})$ as a consequence of \cref{lemma:pullback-unstable-dualizing-sheaf}. The same result also implies that the restriction of $\omega_p^{\Ani_*}$ to any fibre is then equivalent to $\omega_X^{\Ani_*}$. Hence, we obtain the dashed arrow as claimed. 
\end{proof}

Finally, we discuss the compatibility of the versions in families of the unstable dualizing sheaf and the Spivak tangent fibration of a hypercomplete compact ANR homology manifold. We recall first that the dualizing spectrum of a Poincar\'e duality complex exists in families. To that end, let $\PD \subseteq \Ani^\simeq$ be the full subcategory of the groupoid core of anima consisting of the Poincar\'e duality complexes. Then the canonicity of the dualizing spectrum implies that the inclusion $\PD \to \Ani$ admits a canonical lift as indicated in the following diagram:
\[ \begin{tikzcd}
	& \Ani_{/\Sp} \ar[d] \\
	\PD \ar[r] \ar[ur,dashed]& \Ani
\end{tikzcd}\]
Here, the dashed arrow records the Spivak tangent fibration $T_X\colon X \to \Pic(\bS) \subseteq \Sp$ of a Poincar\'e duality complex $X$.
For $X$ a PD complex let us denote by $\aut^{T_X}(X)$ the group of automorphisms of $(X,T_X) \in \Ani_{/\Sp}$. Then the upper dashed diagonal map induces a map 
\[ \Baut(X) \to \Baut^{T_X}(X)\]
which has the following interpretation upon applying the Yoneda embedding: First note that the source classifies maps $E 
\to B$ whose fibres are equivalent to $X$ while the target classifies maps $E\to B$ together with a functor $E \to \Sp$ whose restriction to the fibres is equivalent, in the slice $\Ani_{/\Sp}$, to the Spivak tangent fibration $T_X$ of $X$. The above displayed map then sends a map $p\colon E \to B$ with fibres equivalent to $X$ to the same map $p$ together with the pointwise inverse $T_p$ of the relative dualizing spectrum $D_p$ of $p$. Recall here that $D_p$ is the unique object of $\Fun(E,\Sp)$ such that $p_* \simeq p_!(-\otimes_E D_p)$ as $p^*$-linear functors; this object is compatible with pullbacks in the sense that for a pullback diagram
\[\begin{tikzcd}
	E' \ar[r,"g"] \ar[d,"p'"] & E \ar[d,"p"] \\
	B' \ar[r,"f"] & B
\end{tikzcd}\]
there is a canonical equivalence $g^*(D_p) \simeq D_{p'}$. In particular, if $X$ is a PD complex, $D_p$ is pointwise invertible, so that $T_p=D_p^{-1}$ is again invertible, see e.g.\ \cite[Section 3.1]{Cnossen} for more about the relative dualizing spectrum.

\begin{corollary}\label[corollary]{cor:functorialunstabledual}
	Let $X$ be a compact hypercomplete $\Ani_*$-homotopy manifold.
	There exists a commutative diagram in $\Ani$:
	\[\begin{tikzcd}
		\BHomeo(X) \ar[r] \ar[d] & \Baut^{\omega_X^{\Ani_*}}(X) \ar[d] \\
		\Baut(X) \ar[r] & \Baut^{T_X}(X).
	\end{tikzcd}\]
\end{corollary}

\begin{proof}
	As in \cref{thm:functorialunstabledual}, it suffices to show that the diagram exists in the homotopy category of $\Ani$, and we argue again appealing to the Yoneda embedding. We recall what the functors are that the objects in the displayed square represent. Namely, $\Baut(X)$ represents fibrations $E \to B$ with fibre equivalent to $X$, and $\Baut^{T_X}(X)$ represents fibrations $E \to B$ together with a functor $T \colon E \to \Sp$ whose restriction to any fibre is equivalent to $T_X$ (in the slice $\Ani_{/\Sp}$). The left vertical map simply sends an $X$-fibre bundle to the underlying $X$-fibration, and the lower horizontal arrow associates to an $X$-fibration $p\colon E \to B$ the same $X$-fibration together with the relative tangent spectrum of $p$. The proof of the corollary is hence finished once we show that $\Sigma^\infty \omega_p^{\Ani_*} = \omega_p^{\Sp}=\omega_p$ is canonically equivalent to $T_p$, that is, a relative form of \cref{thm:loccontrlocconstdualisPD}. 
	
	To that end, we first note that for a $X$-fibre bundle $p \colon E \to B$, the following three conditions are equivalent and valid, namely, that either of $\pfp{p},\pf{p},\pfs{p}$ preserve locally constant sheaves. Indeed, since $X$ is assumed compact, $p$ is proper and we have $\pfp{p}=\pf{p}$, moreover, $p$ is $\Sp$-smooth, so we have $\pfs{p} = \pfp{p}(-\otimes \omega_p)$ and $\omega_p$ is invertible. Now, to see that $\pfs{p}$ preserves locally constant sheaves, by base-change it suffices to show this locally. Since $p$ is locally trivial, we may hence assume that $p$ is the projection $X \times B \to B$. Under the K\"unneth isomorphism, we have the following commutative diagram
	\[\begin{tikzcd}
		\Sh{X\times B}{\Sp}^{\mathrm{lc}} \ar[r,hook, "\pb{\eta}\otimes\id"] \ar[d,hook,"\pb{\eta}"] & \Sh{X}{\Sp}\otimes \Sh{B}{\Sp}^{\mathrm{lc}} \ar[r,"\pfs{a}\otimes \id"] \ar[r] &  \Sh{B}{\Sp}^{\mathrm{lc}} \ar[d,"\pb{\eta}", hook] \\
		\Sh{X\times B}{\Sp} \ar[rr,"\pfs{p}"] && \Sh{B}{\Sp}
	\end{tikzcd}\]
	where $a\colon X \to \ast$ is the unique map. This shows that $\pfs{p}$ indeed preserves locally constant sheaves. 
	
	Having this, similarly as for the proof of \cref{thm:loccontrlocconstdualisPD}, and denoting by $\varphi \colon E \to B$ the map of anima underlying $p$ we then find equivalences of $\varphi^*$-linear functors $\Fun(E,\Sp) \simeq \Sh{E}{\Sp}^{\mathrm{lc}} \to \Fun(B,\Sp) \simeq \Sh{B}{\Sp}^{\mathrm{lc}}$ as follows
	\[ \pf{p}(\pb{\eta}(-)\otimes\omega_p) \simeq \pfs{p}(\pb{\eta}(-)) = \varphi_!(-) \quad \text{ and } \quad \pf{p}(\pb{\eta}(-)) \simeq \varphi_!(-\otimes D_p). \]
	Since $\omega_p$ is locally constant and $\eta$ is symmetric monoidal, we again find
	\[ \varphi_!(-) \simeq \pf{p}(\eta^*(-)\otimes \omega_p) = \pf{p}(\eta^*(-\otimes \omega_p)) = \varphi_!(-\otimes \omega_p \otimes D_p)\]
	so by uniqueness, we deduce $\omega_p = D_p^{-1} = T_p$ as claimed.
\end{proof}

%

\subsection{Homotopy manifolds with conical singularities}

In this subsection, we focus our attention on a more geometric family of $\Ani_*$-homotopy manifolds, that we call \textit{homotopy manifolds with conical singularities}. We prove a generalization of a theorem of Siebenmann, stating that any such homotopy manifold is in fact a topological manifold. In what follows, we make use of the notion of a \textit{$C^0$-stratified space}. We refer to \cite[Definition 2.1.15]{AyalaFrancisTanaka2017} for a definition of $C^0$-stratified spaces.

\begin{definition}\label[definition]{def:conical-singularities}
    A \textit{homotopy manifold with conical singularities} is a second countable Hausdorff $C^0$-stratified space $X\rightarrow P$ such that, for each compact $C^0$-stratified space $Z\rightarrow Q$ and stratified open embedding $\rnum^n\times C(Z)\hookrightarrow X$, we have that $Z$ is homotopy equivalent to a sphere. 
\end{definition}


\begin{lemma}
    Any homotopy manifold with conical singularities is an $\Ani_*$-homotopy manifold.
\end{lemma}

\begin{proof}
	By \cite[Corollary A.3]{volpe2022verdier}, we know that $X$ is $\Ani$-locally contractible. Hence we need to show that, for any $x\in X$,  $\sh(X | x)$ is homotopy equivalent to a sphere. By excision and \cite[Lemma 2.2.2, Remark 2.2.3]{AyalaFrancisTanaka2017}, we may assume that $X=\rnum^n\times C(Z)$, where $Z\to Q$ is compact and $C^0$-stratified, and that $x$ lies in the initial stratum. Since $\sh(X)\simeq\ast$, to conclude the proof it suffices to show that $\sh(X\setminus\{x\})\simeq S^m$ for some $m\in\mathbb{N}$. Covering $X\setminus\{x\}$ by $\rnum^n\setminus\{0\}\times C(Z)$ and $\rnum^n\times \rnum_{>0}\times Z$, we see that $\sh(X\setminus\{x\})$ is the pushout in $\Ani$ of the span $\sh(S^{n-1})\leftarrow \sh(S^{n-1})\times \sh(Z)\rightarrow \sh(Z)$. By assumption, there exists $d\in\mathbb{N}$ such that $Z$ is homotopy equivalent to $S^d$. Therefore we see that $\sh(X\setminus\{x\})\simeq\sh(S^{n+d})$ as desired.
\end{proof}

\begin{proposition}\label[proposition]{htpymanwithcones:sufficescovering}
    Let $X\rightarrow P$ be a locally compact second countable Hausdorff stratified space. Then $X$ is a homotopy manifold with conical singularities if and only if $X$ admits an open covering $\{U_i\}_{i\in I}$ such that $U_i$ is isomorphic as a stratified space to one of the form $\rnum^{n_i}\times C(Z_i)$, where $Z_i$ is compact, $C^0$-stratified and homotopy equivalent to a sphere.
\end{proposition}
\begin{proof}
    Clearly, any homotopy manifold with conical singularities admits an open covering as the one in the statement of \cref{htpymanwithcones:sufficescovering}. 
    
    Conversely, suppose that $X\rightarrow P$ is a stratified space admitting such an open covering. Then $X$ is in particular $C^0$-stratified. Let $\rnum^n\times C(Z)\hookrightarrow X$ be any stratified open embedding, with $Z$ compact and $C^0$-stratified. By the topological invariance of cones (see \cite{kwun1964uniqueness}, \cite[Corollary 4.12]{siebenmann1972deformation}), there exists an $i\in I$ and a stratified homeomorphism $$\rnum^n\times C(Z)\cong\rnum^{n_i}\times C(Z_i).$$ By removing the initial stratum on both sides, we obtain that $Z$ is homotopy equivalent to $Z_i$, and therefore homotopy equivalent to a sphere. Thus, we conclude that $X\rightarrow P$ is a homotopy manifold with conical singularities.
\end{proof}

\begin{corollary}
    Suppose $X$ is a homotopy manifold in the sense of \cite{cohen1970homeomorphisms}. Then $X$ is a homotopy manifold with conical singularities.
\end{corollary}
\begin{proof}
    By definition, $X$ has a triangulation, and so it admits a conical stratification given by the poset of faces (see \cite[Definition A.6.7, Proposition A.6.8]{lurie2017higher}). More explicitely, for each $k$-simplex $\sigma$ of $X$, one checks that there is a stratified open embedding 
    $$\rnum^{n-k-1}\times C(\text{Link}_{\sigma})\hookrightarrow X,$$
    where $n$ is the dimension of $X$ and $\text{Link}_{\sigma}$ denotes the link of the simplex $\sigma$. These open subsets cover $X$, and each link is again triangulated and of smaller dimension. Therefore, one sees that the given stratification of $X$ is $C^0$. By assumption, we also know that $\text{Link}_{\sigma}$ is homotopy equivalent to a sphere. Hence, by \cref{htpymanwithcones:sufficescovering}, we deduce that $X$ is a homotopy manifold with conical singularities.
\end{proof}

The following theorem is a generalization of \cite[Theorem A']{siebenmann1970non}.

\begin{theorem}\label[theorem]{thm:htpymanwithconeareman}
    Let $X\rightarrow P$ be a homotopy manifold with conical singularities. Then $X$ is a topological manifold.
\end{theorem}

\begin{proof}
    Let $d$ be the depth of $X\rightarrow P$. We will prove the theorem by induction on $d$.

    Assume that $d=0$. Then, by \cite[Lemma 1.23]{nocera2023whitney}, $X$ is a topological manifold, and so there is nothing to prove.

    Assume that $d>0$. Let $\rnum^n\times C(Z)\hookrightarrow X$ be a stratified open embedding, with $Z$ compact $C^0$-stratified. We claim that $Z$ is itself a homotopy manifold with conical singularities. Indeed, let $\rnum^k\times C(T)$ be a conical chart in $Z$. By crossing with $\rnum$, we get a stratified open embedding $\rnum^{k+1}\times C(T)\hookrightarrow\rnum\times Z\hookrightarrow C(Z)$. By further crossing with $\rnum^{n-k}$, we get stratified open embeddings
    $$\rnum^{n+1}\times C(T)\hookrightarrow\rnum^{n+1}\times Z\hookrightarrow\rnum^n\times C(Z)\hookrightarrow X.$$
    Therefore, since $X$ is a homotopy manifold with conical singularities, we get that $T$ must be homotopy equivalent to a sphere, and hence $Z$ is a homotopy manifold with conical singularities.

    Since the depth of $Z$ is strictly smaller than $d$, by the inductive hypothesis we get that $Z$ is a topological manifold. Therefore, by the topological Poincar\'e conjecture (see \cite{newman1966engulfing} for dimension $>4$, \cite{freedman1982topology}, \cite{behrens2021disc} for dimension $4$ and \cite{perelman2002entropy} for dimension $3$), $Z$ must be homeomorphic to a sphere. Thus we have a homeomorphism $\rnum^n\times C(Z)\cong\rnum^{n+1+\text{dim}(Z)}$, which provides the sought euclidean chart for $X$.
\end{proof}

\appendix

\section{Local connectedness}
In this appendix, we explore a truncated version of local contractiblity, which, in the context of ordinary topoi, was introduced by Johnstone \cite{Johnstone, Johnstone2}. We show that this toposic notion of local connectedness agrees with the point-set topological one. This result is surely known, but we include a proof because we were unable to find a reference dealing with it. We use this result to deduce that $R$-locally contractible spaces are locally connected, an observation that we need in the main body of the paper.

We recall that for a topological space $Y$, one may consider the equivalence relation on $Y$ defined by declaring that two points are equivalent if and only if they lie in the same connected component. The quotient topology then endows the set $\pi(Y)$ of equivalence classes (i.e.\ of connected components of $Y$) with a topology. This turns out to be totally disconnected, and $Y \to \pi(Y)$ is the initial map from $Y$ to a totally disconnected space. Moreover, $Y$ is locally connected if and only if for all opens $V \subseteq Y$, the totally disconnected topology on $\pi(V)$ is in fact the discrete topology. 

\begin{proposition}\label[proposition]{toposlocconn=locconn}
	Let $X$ be a topological space. Then the functor $\pb{a}_{\Set}\colon \Set \to \Sh{X}{\Set}$ admits a left adjoint if and only if $X$ is locally connected.
\end{proposition}
\begin{proof}
	If $X$ is locally connected, the cosheaf $U \mapsto \pi(U)$ on $X$ extends by colimits to a functor $\pfs{a}\colon \Sh{X}{\Set} \to \Set$ which one checks to be left adjoint to $\pb{a}$.

	Conversely, by the discussion above, it will suffice to prove that for any open inclusion $j\colon U\to X$, the quotient topology on $\pi(U)$ is discrete. However, since $j^*\colon \Sh{X}{\Set}\to \Sh{U}{\Set}$ admits a left adjoint, in fact it suffices to prove that $\pi(X)$ is discrete. By assumption, there exists a set $I$\footnote{Namely, the left adjoint of $a^*$ applied to $a^*(\ast)$.} and a bijection
	$$\Sec{X}{\pb{a}_{\Set}(M)}\cong\Hom{\Set}{I}{M}$$
	which is natural in $M$. Recall that there is a natural isomorphism $$\Sec{X}{\pb{a}_{\Set}(M)}\cong\Hom{\text{Top}}{X}{M^{\delta}},$$ where $M^{\delta}$ is the set $M$ endowed with the discrete topology. Therefore, since any discrete topological space is totally disconnected, we get an isomorphism
	$$\alpha_M\colon \Hom{\Set}{I}{M}\xrightarrow[]{\cong}\Hom{\text{Top}}{\pi(X)}{M^{\delta}}$$
	which is natural in $M$. The naturality of $\alpha$ implies in particular that, for any function $f\colon I\rightarrow M$, the following square commutes.
	$$
	\begin{tikzcd}
		\Hom{\Set}{I}{I} \arrow[d, "f\o-"'] \arrow[r, "\alpha_I"] & \Hom{\text{Top}}{\pi(X)}{I^{\text{disc}}} \arrow[d, "f\o-"] \\
		\Hom{\Set}{I}{M} \arrow[r, "\alpha_M"]                   & \Hom{\text{Top}}{\pi(X)}{M^{\text{disc}}}.                 
	\end{tikzcd}
	$$
	As a consequence, by evaluating $\alpha_I$ at the identity of $I$, we obtain a continuous function $$p\colon \pi(X)\rightarrow I^{\delta}$$ such that the map
	\[-\o p \colon \Hom{\Set}{I}{M}\rightarrow\Hom{\text{Top}}{\pi(X)}{M^\delta}\]
	given by precomposing with $p$ coincides with $\alpha_M$. We will show that $p$ is a homeomorphism.
	
	Since any continuous map with discrete target is automatically open, it suffices to show that $p$ is bijective. By taking $M=\{0,1\}$ in the above bijection, we see that the function
	$$
	\begin{tikzcd}[row sep=tiny]
		\{J\subseteq I\} \arrow[r] & \{W\subseteq\pi(X)\mid \text{$W$ is closed and open}\} \\
		J \arrow[r, maps to]       & p^{-1}(J)                                               
	\end{tikzcd}
	$$
	is an order preserving bijection. In particular, $p^{-1}(J)=\emptyset$ if and only if $J=\emptyset$, which implies that $p$ is surjective.
	
	To show that $p$ is injective, i.e.\ that $|p^{-1}(\{i\})|=1$ for every $i \in I$, it suffices to show that $p^{-1}(\{i\})$ is connected, since $\pi(X)$ is totally disconnected. To see this, suppose $A \subseteq p^{-1}(\{i\})$ is a non-empty, closed and open subset. Then $A \subseteq \pi(X)$ is again non-empty, closed and open. Under the above bijection, $A$ corresponds to a non-empty subset of $\{i\}$, so it follows that $A = p^{-1}(\{i\})$ as needed.
\end{proof}

The next result shows that $R$-locally contractible spaces are locally connected.

\begin{lemma}\label[lemma]{Sploccontr=>Setlocconn}
	Let $X$ be a topological space and $R$ a connective $\bE_1$-ring spectrum such that the functor $\pb{a}_{R}\colon \Mod_R \to \Sh{X}{\Mod_R}$ admits a left adjoint. Then $\pb{a}_{\Set}\colon \Set \to \Sh{X}{\Set}$ also admits a left adjoint.
\end{lemma}
\begin{proof}
	Since $\pb{a}_\Set$ preserves finite limits and all categories are presentable, it suffices to show that $\pb{a}_\Set$ preserves arbitrary products.
	First, we note that the functor $\pb{a}_{R} \colon \Mod_R \to \Sh{X}{\Mod_R}$ it $t$-exact and induces the constant sheaf functor $\pb{a}_{\pi_0(R)}\colon \Mod_{\pi_0(R)}^\heartsuit \to \Sh{X}{\Mod_{\pi_0(R)}^\heartsuit}$ on hearts. Since the hearts $\Mod_{\pi_0(R)}^\heartsuit$ and $\Sh{X}{\Mod_{\pi_0(R)}^\heartsuit}$ are closed under products in $\Mod_R$ and $\Sh{X}{\Mod_R}$ respectively, we deduce that $\pb{a}_{\pi_0(R)}$ preserves products. Moreover, from the commutative diagram
	\[\begin{tikzcd}
		\Mod_{\pi_0(R)}^\heartsuit \ar[r,"\pb{a}_{\pi_0(R)}"] \ar[d] & \Sh{X}{\Mod_{\pi_0(R)}^\heartsuit} \ar[d] \\
		\Set \ar[r,"\pb{a}_\Set"] & \Sh{X}{\Set}
	\end{tikzcd}\]
	whose vertical maps are the forgetful maps, we deduce that the functor $\pb{a}_{\Set}$ preserves products of sets which underlie a discrete $\pi_0(R)$-module. 
	
	Now, let $I$ be a set and for each $i \in I$, let $S_i$ be another set. We wish to show that the canonical map $\pb{a}(\prod_{i \in I} S_i) \to \prod_{i \in I} \pb{a}(S_i)$ is an isomorphism. First, we argue that this is the case if there exists an $i \in I$ such that $S_i$ is empty. In this case, the left hand side is $\pb{a}(\emptyset)$, which takes the value $\emptyset$ on any non-empty open of $X$ and a singleton on the empty subset of $X$. Since sheaves are closed under products in presheaves, the right hand side takes the same values, and consequently the map is necessarily an isomorphism. 
	
	We may finally assume that for each $i \in I$, the set $S_i$ is not empty. In this case, consider for $i \in I$ the free $\pi_0(R)$-module $M_i$ on the set $S_i$. Since all $S_i$ are non-empty, we find that $S_i$ is a retract of $M_i$. Consequently, we deduce that the map in question is a retract of the map $\pb{a}(\prod_{i \in I} M_i) \to \prod_{i \in I} \pb{a}(M_i)$ which is an isomorphism as we have discussed above. This finishes the proof of the lemma.
\end{proof}

\bibliographystyle{amsalpha}
\bibliography{pdhm-2}

\end{document}